\documentclass[12pt,a4paper]{amsart}
\usepackage{amsfonts}
\usepackage{amsthm}
\usepackage{amsmath}
\usepackage{amscd}
\usepackage{amssymb}
\usepackage[latin2]{inputenc}
\usepackage{t1enc}
\usepackage[mathscr]{eucal}
\usepackage{indentfirst}
\usepackage{graphicx}
\usepackage{graphics}
\usepackage{pict2e}
\usepackage{epic}
\numberwithin{equation}{section}
\usepackage[margin=2.9cm]{geometry}
\usepackage{epstopdf} 
\usepackage{enumerate}
\usepackage[shortlabels]{enumitem}
\usepackage{verbatim}
\usepackage{cite}
\usepackage{mathtools}
\usepackage{mathrsfs}
\usepackage{xcolor}
\usepackage{hyperref}
\usepackage[normalem]{ulem}

\theoremstyle{plain}
\newtheorem{Th}{Theorem}[section]
\newtheorem{Lemma}[Th]{Lemma}
\newtheorem{Cor}[Th]{Corollary}
\newtheorem{Prop}[Th]{Proposition}
\newtheorem{Claim}[Th]{Claim}

\newtheoremstyle{named}{}{}{\itshape}{}{\bfseries}{.}{.5em}{\thmnote{#3}}
\theoremstyle{named}

\theoremstyle{definition}
\newtheorem{Def}[Th]{Definition}

\newtheorem{Rem}[Th]{Remark}

\newcommand{\area}{{\rm{area}}}
\newcommand{\vol}{{\rm{vol}}}

\newcommand{\ww}{\mathtt{w}}
\newcommand{\dvol}{\operatorname{dvol}}
\definecolor{ao(english)}{rgb}{0.0, 0.5, 0.0}

\renewcommand\Re{\operatorname{Re}}
\renewcommand\Im{\operatorname{Im}}

\begin{document}

\title[Stability of Ricci flow on hyperbolic 3-manifolds of finite volume]{On the stability of Ricci flow on hyperbolic 3-manifolds of finite volume}

\author{Ruojing Jiang}
\address{Massachusetts Institute of Technology, Department of Mathematics, Cambridge, MA 02139} 
\email{ruojingj@mit.edu}

\author{Franco Vargas Pallete}
\address{School of Mathematical and Statistical Sciences, Arizona State University, Tempe, AZ 85287}
\email{fevargas@asu.edu}

 \subjclass[2020]{} 
 \date{}

\begin{abstract}  
On a hyperbolic 3-manifold of finite volume, we prove that if the initial metric is sufficiently close to the hyperbolic metric $h_0$, then the normalized Ricci-DeTurck flow exists for all time and converges exponentially fast to $h_0$ in a weighted H{\"o}lder norm. A key ingredient of our approach is the application of interpolation theory.

Furthermore, this result is a valuable tool for investigating minimal surface entropy, which quantifies the growth rate of the number of closed minimal surfaces in terms of genus. We explore this in \cite{Jiang-VargasPallete_entropy}.
\end{abstract}

\maketitle

\section{Introduction}
The Ricci flow, introduced by Hamilton in his seminal paper~\cite{Hamilton1982}, evolves a Riemannian metric \( h(t) \) on a manifold \( M \) according to the evolution equation:
$$
\frac{\partial}{\partial t} h(t) = -2 Ric(h(t)),
$$
where $Ric(h(t))$ denotes the Ricci curvature of the evolving metric. The flow tends to smooth out geometric irregularities and, under appropriate conditions, guides the metric toward canonical forms. Hamilton's foundational contributions initiated a geometric analysis program that culminated in Perelman's resolution of the Poincar{\'e} and Geometrization Conjectures using Ricci flow with surgery~\cite{Perelman2002, Perelman2003a, Perelman2003b}.

A central question in the study of Ricci flow on nonpositively curved manifolds is the long-time behavior of solutions and the stability of special metrics under perturbation. In particular, one asks whether the Ricci flow starting near a special metric, such as Einstein metrics, will converge back to such a structure. This question has driven extensive work on dynamical stability, especially for compact manifolds and certain symmetric noncompact ones.

Guenther, Isenberg, and Knopf~\cite{Guenther-Isenberg-Knopf} established the dynamical stability of compact Ricci-flat metrics. Their approach employed maximal regularity theory for parabolic equations, as developed by Da Prato and Grisvard~\cite{DaPratoGrisvard1975}, and center manifold theory in the framework of Simonett~\cite{Simonett}. They showed that starting from a metric in a little H{\"o}lder $\Vert \cdot\Vert _{1+\eta}$ neighborhood of a flat metric $h_{flat}$ on a torus $T^n$, the Ricci flow converges exponentially fast in the $\Vert \cdot\Vert _{2+\rho}$ norm to a flat metric on $T^n$ (possibly different from $h_{flat}$). 
Building on similar tools, Knopf~\cite{Knopf} studied the convergence and stability of $\mathbb{R}^N$-invariant solutions, while Knopf and Young~\cite{Knopf-Young} analyzed the case of closed hyperbolic 3-manifolds under both the normalized Ricci and cross-curvature flows. Wu~\cite{Wu} extended these ideas to complex hyperbolic spaces and explored the exponential attractivity to the complex hyperbolic metric under perturbation.

Other approaches to Ricci flow stability include Ye's work on convergence under Ricci pinching conditions~\cite{Ye}, \v{S}e\v{s}um's analysis of the stability of K{\"a}hler-Einstein metrics on K3 surfaces~\cite{Sesum2006}, and Li and Yin \cite{LiYin2010}, who studied the stability of normalized Ricci flow near hyperbolic metrics in dimensions $n \geq 6$.
Schn{\"u}rer, Schulze, and Simon~\cite{SchnurerSchulzeSimon2010} demonstrated stability for real hyperbolic spaces in dimensions $n \geq 4$ under the scaled Ricci-harmonic map heat flow, while Hu, Ji, and Shi~\cite{HuJiShi2020} proved the stability of strictly stable conformally compact Einstein metrics in dimensions $n \geq 3$. For noncompact finite-volume manifolds, Ji, Mazzeo, and \v{S}e\v{s}um \cite{JiMazzeoSesum2009} analyzed Ricci flow stability on hyperbolic surfaces with cusps.

In this paper, we focus on the Ricci flow on hyperbolic 3-manifolds of finite volume. Similar to the compact case, it is natural to ask whether the Hamilton-Perelman results can be extended to manifolds with cusps. 
We are interested in the stability of the Ricci flow at its fixed point, specifically the hyperbolic metric.
Bessi{\`e}res, Besson, and Maillot established the construction of Ricci flow with a specific version of surgery on cusped manifolds in \cite{Bessieres-Besson-Maillot}, called \emph{Ricci flow with bubbling-off}, with the assumption that the initial metric has a cusp-like structure. For the second question, their work indicates that, after a finite number of surgeries, the solution converges smoothly to the hyperbolic metric on balls of radius $R$ for all $R>0$ as $t$ approaches infinity. However, outside these balls, it may be asymptotic to a different hyperbolic structure on the cusps, meaning that the convergence need not be global on $M$ because the cusps allow for trivial Einstein variations. Bamler \cite{Bamler} showed that if the initial metric is a small $C^0$ perturbation of the hyperbolic metric, then the Ricci flow converges on any compact sets and remains asymptotic to the same hyperbolic structure for all time.

We will explore a more quantitative version of the stability of hyperbolic metrics on finite-volume hyperbolic 3-manifolds under the normalized Ricci-DeTurck flow. We embed a Ricci flow ray into a bigger Banach space that contains trivial Einstein variations. Our strategy builds on maximal regularity theory and interpolation techniques, following the approach of Angenent~\cite{Angenent_1990}, which extends the work of Da Prato and Grisvard. By working with a pair of densely embedded Banach spaces and an operator that generates a strongly continuous analytic semigroup, we obtain maximal regularity for solutions of the normalized Ricci-DeTurck flow. This framework enables us to derive exponential convergence to the hyperbolic metric, with optimal decay rate given by the spectral estimate of the linearized operator.

\subsection{Main result}
Suppose that $M$ is a hyperbolic 3-manifold of finite volume, equipped with the hyperbolic metric $h_0$. Due to the presence of cusp structures, the standard H{\"o}lder norm, which is typically used to study the stability of the Ricci flow in compact manifolds, is not applicable. The specific reason for this is explained later in Remark~\ref{rem_isom}. To address this issue, we introduce a weighted modification of the norm.

Given a weight parameter $\lambda\in (0,1]$ and a spatial parameter $s\geq 0$. For every $k\in\mathbb{N}$ and $\rho\in (0,1)$,
let $\mathfrak{h}_{\lambda,s}^{k+\rho}$ denote the weighted little H{\"o}lder space on $M$, defined by applying an exponential weight $e^{-\lambda r(x)}$ if $\lambda\in (0,1)$ and $(r(x)+1)e^{-r(x)}$ if $\lambda=1$ in the cusps. Here $r(x)\geq 0$ represents the distance from a point $x$ in a cusp to the boundary of the thick part $M(s)$, that is $\cup_jT_j\times \{s\}$.
Set $\mathcal{X}_0=\mathfrak{h}_{\lambda,s}^{0+\rho}$ and $\mathcal{X}_1=\mathfrak{h}_{\lambda,s}^{2+\rho}$.
Additionally, for a fixed $\alpha\in (0,1)\setminus\{\frac{1-\rho}{2},1-\frac{\rho}{2}\}$,
we define $\mathcal{X}_\alpha:=(\mathcal{X}_0,\mathcal{X}_1)_\alpha=\mathfrak{h}_{\lambda,s}^{2\alpha+\rho}$, which represents the continuous interpolation space between $\mathcal{X}_0$ and $\mathcal{X}_1$. The precise definition is provided in Definition~\ref{def_little_holder}.

We will prove the following stability result for cusped hyperbolic 3-manifolds using the interpolation theory.

\begin{Th}\label{thm_main}
    Let $(M,h_0)$ be a hyperbolic 3-manifold of finite volume. Given $\lambda\in (0,1]$. For every
    $\omega \in (0,\lambda(2-\lambda))$, 
    There exist $\rho_0, c>0$, such that if $g$ is a smooth metric on $M$ with $$\Vert h-h_0\Vert _{C^0(M)}<\rho_0,$$
then the solution $g(t)$ of the normalized Ricci-DeTurck flow \eqref{DRF} starting at $g(0)=g$ exists for all time. 
Moreover, we have \begin{equation*}
    \Vert g(t)-h_0\Vert _{\mathcal{X}_{1}}
    \leq \frac{c}{(t-1)^{1-\alpha}} e^{-\omega t}\Vert g-h_0\Vert _{C^0(M)}, \quad \forall t>1.
\end{equation*}
\end{Th}

\subsection{Application for exponential convergence}
Applying the theorem above, we will present the following application in \cite{Jiang-VargasPallete_entropy}.

On a closed hyperbolic $n$-manifold $M$ ($n\geq3$), Hamenst\"{a}dt \cite{Hamenstadt90} studied the topological entropy of the geodesic flow and proved that the hyperbolic metric attains its minimum among all metric in $M$ with sectional curvature not exceeding $-1$. Recently, Calegari, Marques, and Neves \cite{Calegari-Marques-Neves} introduced the concept of minimal surface entropy of closed hyperbolic 3-manifolds, building on the construction and calculation of surface subgroups by Kahn and Markovic \cite{Kahn-Markovic} \cite{Kahn-Markovic_surface_subgroup}, and proved the analogous statement to the one in \cite{Hamenstadt90}. The minimal surface entropy $E(h)$ measures the exponential asymptotic growth of the number (ordered by area) of $\epsilon$-almost totally geodesic essential minimal surfaces in $M$ with respect to a metric $h$, while sending $\epsilon\rightarrow0$. This shifts the focus from one-dimensional objects (geodesics) to two-dimensional minimal surfaces.

For a closed hyperbolic 3-manifold $M$, Lowe and Neves \cite{Lowe-Neves} utilized the exponential convergence of the normalized Ricci-DeTurck flow to the hyperbolic metric $h_0$ to prove the following result. If $h$ is a Riemannian metric on $M$ with scalar curvature $R_h\geq -6$, then $E(h)\leq E(h_0)$, where the asymptotic counting is done for surfaces that equidistribute in the limit as $\epsilon\rightarrow0$ (i.e. their induced Radon probability on the frame bundle converges vaguely to the Lebesgue measure). Equality holds if and only if $h$ is isometric to the hyperbolic metric $h_0$.

In \cite{Jiang-VargasPallete_entropy} we extend this result for finite volume hyperbolic $3$-manifolds by applying Theorem~\ref{thm_main}. This comparison inequality is stated for \emph{weakly cusped} metrics $h$ in a hyperbolic $3$-manifold $(M,h_0)$ (see \cite[Definition 1.3]{Jiang-VargasPallete_entropy} for more details) as follows.

\begin{Th}[Theorem C, \cite{Jiang-VargasPallete_entropy}]
Let $(M,h_0)$ be a hyperbolic $3$-manifold of finite volume, and assume that it is infinitesimally rigid. Let $h$ be a weakly cusped metric on $M$. If the scalar curvature of $h$ is greater than or equal to $-6$, then \begin{equation*}
    E(h)\leq E(h_0).
\end{equation*}
Furthermore, suppose that $h$ is asymptotically cusped of order at least two, and it satisfies $\Vert Rm(h)\Vert_{C^{1}(M)}<\infty$. Then the equality holds if and only if $h$ is isometric to $h_0$.
\end{Th}

\begin{Th}[Theorem D, \cite{Jiang-VargasPallete_entropy}]
Let $(M,h_0)$ be a hyperbolic $3$-manifold of finite volume, and let $h$ be a weakly cusped metric on $M$ that satisfies the following conditions.
\begin{itemize}
    \item $\Vert h-h_0\Vert_{C^0(M)}\leq\epsilon$ for a given constant $\epsilon>0$,
    \item $h$ is asymptotically cusped of order at least two with $\Vert Rm(h)\Vert_{C^{1}(M)}<\infty$.
\end{itemize} 
If the scalar curvature of $h$ is greater than or equal to $-6$, then \begin{equation*}
    E(h)\leq E(h_0).
\end{equation*}
Furthermore, the equality holds if and only if $h$ is isometric to $h_0$.
\end{Th}

\subsection{Organization}
The paper is organized as follows. Section~\ref{section_background_RF} reviews the necessary background and notation for the Ricci flow, which will be used throughout the paper. In Section~\ref{section_interpolation}, we introduce key preliminaries from interpolation theory that form the foundation for presenting Simonett's stability theorem for autonomous quasilinear parabolic equations, as well as Angenent's existence and uniqueness results for linear equations. Section~\ref{section_Simonett} explores the application of Simonett's theorem to compact manifolds, discusses the challenges that arise in the cusped setting, and outlines a new proof strategy based on Angenent's linear theory.
Section~\ref{section_weight} defines the weighted norms and notation needed for the main results. Section~\ref{section_semigroup generator} then verifies the applicability of the linear theory. Finally, Section~\ref{section_proof} presents the proof of Theorem~\ref{thm_main}. Appendices~\ref{appendix_interpolation}, \ref{sec:L2bounds} provide supplementary proofs for Sections~\ref{section_weight} and \ref{section_semigroup generator}.

\section*{Acknowledegments}
We thank Richard Bamler for helpful suggestions at the start of this project. FVP thanks Yves Benoist for helpful conversations, and thanks IHES for their hospitality during a phase of this work. FVP was partially funded by European Union (ERC, RaConTeich, 101116694)\footnote{Views and opinions expressed are however those of the author(s) only and do not necessarily reflect those of the European Union or the European Research Council Executive Agency. Neither the European Union nor the granting authority can be held responsible for them.}

\section{Background of Ricci flow}\label{section_background_RF}
In this section, we briefly review the tools of Ricci flow used to prove the main theorem and its applications.
\subsection{Normalized Ricci flow and Ricci-DeTurck flow}
The \emph{normalized Ricci flow} on $M$ is defined as \begin{equation}\label{RF}
    \frac{\partial}{\partial t}h(t)=-2Ric(h(t))-4h(t).
\end{equation}
One can easily check that hyperbolic metrics are fixed points of the flow. However, this evolution equation is only weakly parabolic. To achieve strict parabolicity, we introduce the following DeTurck-modified version. Let $Sym^2(T^*M)$ be the space of symmetric covariant $(0,2)$-tensors on $M$, and let $Sym^2_+(T^*M)$ be the subset of positive-definite tensors. Moreover, we denote by $\Omega^1(M):=\Gamma(T^*M)$ the space of differential 1-forms. 
Given a Riemannian metric $h$ on $M$, we use $\delta_h: Sym^2(T^*M)\rightarrow \Omega^1(M)$ to denote the map $\delta_hl=-h^{ij}\nabla_il_{jk}dx^k$. The formal adjoint for the $L^2$ product is denoted by $\delta_h^*:\Omega^1(M)\rightarrow Sym^2(T^*M)$.
Define a map $G: Sym^2_+(T^*M)\times Sym^2(T^*M)\rightarrow Sym^2(T^*M)$ by \begin{equation*}
    G(h,u)=\Big(u_{ij}-\frac{1}{2}h^{km}u_{km}h_{ij}\Big)dx^i\otimes dx^j.
\end{equation*}
And $P:Sym^2_+(T^*M)\times Sym^2_+(T^*M)\rightarrow Sym^2(T^*M)$ is defined by \begin{equation*}
    P_u(h)=-2\delta_h^*\left(u^{-1}\delta_h(G(h,u))\right).
\end{equation*}
Finally, the \emph{normalized Ricci-DeTurck flow} for \eqref{RF} is given by 
\begin{equation}\label{DRF}
    \frac{\partial}{\partial t}h(t)=-2Ric(h(t))-4h(t)-P_{h_0}(h(t)),
\end{equation}
where we set the background metric $u$ to be the hyperbolic metric $h_0$ so that $h_0$ is a fixed point of \eqref{DRF}.
Notice that the right hand side is a strictly elliptic operator known as the \emph{DeTurck operator}.

\subsection{Stability of hyperbolic metrics}\label{subsection_persist}

The following is an application of the results in \cite{Bamler}. Recall that all time existence for small $C^0$ perturbations of a hyperbolic metric follows from \cite[Theorem 1.1]{Bamler}.

\begin{Th}[Stability of hyperbolic metric]\label{thm:Ckstability}
    Let $(M,h_0)$ be a hyperbolic $3$-manifold of finite volume. Let $\epsilon>0$ be so that for any $\Vert g(0)-h_0 \Vert_{C^0(M)} <\epsilon$ we have that the Ricci--DeTurck flow exists for all $t\geq0$. 
    
    Moreover, for any given $k\in\mathbb{N}$, there exist constants $\delta_k\leq \epsilon$ and $C_k>0$ so that the following holds. Let $g(0)$ be a smooth metric on $M$ so that
    \[
    \Vert g(0)-h_0 \Vert_{C^0(M)} \leq \delta_k.
    \]
    Then the Ricci--DeTurck flow $g(t)$ exists satisfies
    \begin{align*}
    \Vert g(t) - h_0 \Vert_{C^k(M)} &\leq C_k \Vert g(0)-h_0 \Vert_{C^0(M)},\quad t\geq 1,\\
    \Vert g(t) - h_0 \Vert_{C^k(M)} &\leq C_k t^{-k/2} \Vert g(0)-h_0 \Vert_{C^0(M)},\quad 0\leq t\leq 1.
    \end{align*}
\end{Th}
\begin{proof}
    From the proof of \cite[Theorem 1.1]{Bamler} (see \cite[Section 6.2]{Bamler}) we have that for $\delta_0$ sufficiently small, there exists $C>0$ so that if $\Vert g(0)-h_0 \Vert_{C^0(M)} \leq \delta_0$ then
    \[
    \Vert g(t)-h_0 \Vert_{C^0(M)} \leq C\Vert g(0)-h_0 \Vert_{C^0(M)}, \quad t\geq 0.
    \]
    By \cite[Corollary 2.7]{Bamler} applied to regions covering $M$ and any $t\geq 0$, we have that for $\delta_k$ sufficiently small, there exists $C_k$ so that if $\Vert g(0)-h_0 \Vert_{C^0(M)} \leq \delta_k$ then
    \[
    \Vert g(s+t)-h_0 \Vert_{C^k(M)} \leq C_k s^{-k/2}\Vert g(t)-h_0 \Vert_{C^0(M)}, \quad 0\leq s\leq 1.
    \]
    The conclusion follows suit.
\end{proof}

\section{Interpolation Theory}\label{section_interpolation}
This section provides a brief overview of interpolation theory. For a more comprehensive treatment, we refer the reader to the textbooks of Lunardi \cite{Lunardi} and Triebel \cite{Triebel}.

Let $\mathcal{X}_0$ and $\mathcal{X}_1$ be two real Banach spaces that are continuously embedded in a linear Hausdorff space $\mathcal{X}$. Such a couple $\{\mathcal{X}_0,\mathcal{X}_1\}$ is called an \emph{interpolation couple}. Let $\{\mathcal{Y}_0,\mathcal{Y}_1\}$ be another interpolation couple, and let $\mathcal{Y}$ be a linear Hausdorff space containing this couple. Let $T$ be a linear operator acting from $\mathcal{X}$ to $\mathcal{Y}$, whose restriction to $\mathcal{X}_i$, where $i=0,1$, is a continuous linear operator from $\mathcal{X}_i$ to $\mathcal{Y}_i$. 
In particular, real interpolation theory, pioneered by J.-L. Lions and J. Peetre \cite{LionsPeetre}, and others, aims to discover constructions, denoted as $F$, that establish new real Banach spaces, denoted as $F({\mathcal{X}_0,\mathcal{X}_1})$, derived from a given pair of real interpolation spaces, ${\mathcal{X}_0,\mathcal{X}_1}$, in a manner ensuring that $F({\mathcal{X}_0,\mathcal{X}_1})$ and $F({\mathcal{Y}_0,\mathcal{Y}_1})$ adhere to specific interpolation properties. Additionally, the theory seeks to outline all spaces within $\mathcal{X}$ and $\mathcal{Y}$ possessing these interpolation properties, along with detailing all possible constructions denoted as $F$.

\subsection{Interpolation spaces}
Let $\{\mathcal{X}_0, \mathcal{X}_1\}$ be an interpolation couple contained in a linear Hausdorff space $\mathcal{X}$. Their intersection $\mathcal{X}_0\cap\mathcal{X}_1$ is a linear subspace of $\mathcal{X}$, and it is a Banach space with the norm \begin{equation*}
    \Vert l\Vert _{\mathcal{X}_0\cap\mathcal{X}_1}:=\max \{\Vert l\Vert _{\mathcal{X}_0},\Vert l\Vert _{\mathcal{X}_1}\}.
\end{equation*}
Additionally, the sum $\mathcal{X}_0+\mathcal{X}_1=\{l_0+l_1:l_0\in \mathcal{X}_0, l_1\in \mathcal{X}_1\}$ is a linear subspace of $\mathcal{X}$, endowed with the norm \begin{equation*}
    \Vert l\Vert _{\mathcal{X}_0+\mathcal{X}_1}:=\inf_{l=l_0+l_1, \,l_i\in \mathcal{X}_i} \left(\Vert l_0\Vert _{\mathcal{X}_0}+\Vert l_1\Vert _{\mathcal{X}_1}\right).
\end{equation*}
The infimum is taken over all representations of $l\in \mathcal{X}_0+\mathcal{X}_1$ in the described way above. As easily seen, $\mathcal{X}_0+\mathcal{X}_1$ is isometric to the quotient space $(\mathcal{X}_0\times\mathcal{X}_1)/D$, where $D=\{(l,-l):l\in \mathcal{X}_0\cap\mathcal{X}_1\}$ is a closed subset of the Hausdorff space $\mathcal{X}$. Therefore, $\mathcal{X}_0+\mathcal{X}_1$ is also a Banach space. 

For a given interpolation couple $\{\mathcal{X}_0, \mathcal{X}_1\}$, a Banach space $\mathcal{E}$ is called an \emph{intermediate space} if \begin{equation*}
    \mathcal{X}_0\cap \mathcal{X}_1\subset \mathcal{E}\subset\mathcal{X}_0+\mathcal{X}_1.
\end{equation*}
Furthermore, let $\mathcal{L}(\mathcal{X}_i)$ be the space of all bounded linear operators from the Banach space $\mathcal{X}_i$ to itself. And let $\mathcal{L}(\mathcal{X}_0)\cap\mathcal{L}(\mathcal{X}_1)$ be the space of all bounded linear operators from $\mathcal{X}_0+\mathcal{X}_1\mapsto \mathcal{X}_0+\mathcal{X}_1$ whose restrictions to $\mathcal{X}_i$ belongs to $\mathcal{L}(\mathcal{X}_i)$, where $i=0,1$.

An \emph{interpolation space between $\mathcal{X}_0$ and $\mathcal{X}_1$} is any intermediate space such that for any $T\in \mathcal{L}(\mathcal{X}_0)\cap\mathcal{L}(\mathcal{X}_1)$, the restriction of $T$ to $\mathcal{E}$ belongs to $\mathcal{L}(\mathcal{E})$.

\subsection{K-method and J-method}\label{section_k_method}
In this subsection, we review two of the real interpolation methods in \cite{Triebel}, the $K$-method and the $J$-method. Both of them give rise to the same interpolation spaces, and both will be helpful for us to understand the Reiteration Theorem~\ref{thm_reiteration}.

For every $l\in \mathcal{X}_0 +\mathcal{X}_1$ and $t>0$, set \begin{equation*}
    K(t,l)=K(t,l;\mathcal{X}_0,\mathcal{X}_1):=\inf_{l=l_0+l_1,\, l_i\in\mathcal{X}_i}\left(\Vert l_0\Vert _{\mathcal{X}_0}+t\Vert l_1\Vert _{\mathcal{X}_1}\right).
\end{equation*}
For each $t$, it defines an equivalent norm for the space $\mathcal{X}_0+\mathcal{X}_1$. 
\begin{Def}\label{def_real/cts_interpolation}
    Let $0<\theta<1$, $1\leq p\leq \infty$, and define the following \emph{real interpolation spaces} between $\mathcal{X}_0$ and $\mathcal{X}_1$: 
        \item \begin{equation*}
            (\mathcal{X}_0,\mathcal{X}_1)_{\theta,p}:=\left\{l\in \mathcal{X}_0+\mathcal{X}_1: t\mapsto t^{-\theta}K(t,l)\in L^p_*(0,\infty)\right\},
        \end{equation*} 
        where $L^p_*$ is the $L^p$ space with respect to the measure $dt/t$. Note that the $L^\infty_*$ space coincides with the standard $L^\infty$ space. 
        The norm of $l\in (\mathcal{X}_0,\mathcal{X}_1)_{\theta,p}$ is given by \begin{equation*}
        \Vert l\Vert _{(\mathcal{X}_0,\mathcal{X}_1)_{\theta,p}}:= \Vert t^{-\theta}K(t,l)\Vert _{L^p_*(0,\infty)}.
        \end{equation*}
        Moreover, the \emph{continuous interpolation space} between $\mathcal{X}_0$ and $\mathcal{X}_1$ is defined as follows. \begin{equation*}
            (\mathcal{X}_0,\mathcal{X}_1)_{\theta}:= \left\{l\in \mathcal{X}_0+\mathcal{X}_1: \lim_{t\rightarrow 0^+}t^{-\theta}K(t,l)=\lim_{t\rightarrow\infty}t^{-\theta}K(t,l)=0\right\}.
        \end{equation*}
\end{Def}
Observe that the function $K(t,x)$ is continuous in terms of $t$, thus $(\mathcal{X}_0,\mathcal{X}_1)_{\theta}$ is a closed subspace of $(\mathcal{X}_0,\mathcal{X}_1)_{\theta,\infty}$ and it is endowed with the $(\mathcal{X}_0,\mathcal{X}_1)_{\theta,\infty}$-norm.

An important application of the $K$-method is stated in the following lemma (Corollary 1.7 of \cite{Lunardi} and Theorem 1.3.3 (g) of \cite{Triebel}). 
\begin{Lemma}\label{lemma_interpolation_inequality}
    Let $\{\mathcal{X}_0, \mathcal{X}_1\}$ be an interpolation couple. Given any $0<\theta<1$, there exists a constant $c>0$ such that \begin{equation*}
         \Vert l\Vert _{(\mathcal{X}_0,\mathcal{X}_1)_{\theta}}\leq c \,\Vert l\Vert ^{1-\theta}_{\mathcal{X}_0}\Vert l\Vert ^\theta_{\mathcal{X}_1},\quad \forall l\in \mathcal{X}_0\cap\mathcal{X}_1.
    \end{equation*}
\end{Lemma}

As an analogue to the $K$-method that defines the real and continuous interpolation spaces, we introduce the definition of the $J$-method. 
\begin{equation*}
    J(t,l)=J(t,l;\mathcal{X}_0,\mathcal{X}_1):= \max\left(\Vert l\Vert _{\mathcal{X}_0},t\Vert l\Vert _{\mathcal{X}_1}\right),\quad \forall l\in \mathcal{X}_0\cap \mathcal{X}_1.
\end{equation*}
According to Section 1.6.1 of \cite{Triebel}, the real and continuous interpolation spaces defined using $J(t,l)$ are equivalent to those defined in Definition~\ref{def_real/cts_interpolation}.

\subsection{Reiteration Theorem}

\begin{Def}\label{def_K_J_classes}
Let $\{\mathcal{X}_0, \mathcal{X}_1\}$ be an interpolation couple, and let $\mathcal{E}$ be an interpolation space between $\mathcal{X}_0$ and $\mathcal{X}_1$, so we have
    $\mathcal{X}_0\cap \mathcal{X}_1\subset \mathcal{E}\subset\mathcal{X}_0+\mathcal{X}_1$.
    Set $0\leq \theta \leq 1$. \begin{itemize}
        \item We say that $\mathcal{E}$ belongs to the class $K_\theta=K_\theta(\mathcal{X}_0,\mathcal{X}_1)$ between $\mathcal{X}_0$ and $\mathcal{X}_1$ if  one of the following equivalent conditions holds: \begin{enumerate}[(1)]
            \item \begin{equation*}
                \mathcal{X}_0\cap\mathcal{X}_1\subset \mathcal{E}\subset (\mathcal{X}_0,\mathcal{X}_1)_{\theta,\infty},
            \end{equation*} 
            see Definition 1.10.1 of \cite{Triebel};

            \item There exists $k>0$ such that \begin{equation*}
                K(t,l)\leq k\,t^\theta \Vert l\Vert _{\mathcal{E}}, \quad \forall l\in \mathcal{E}, \,t>0,
            \end{equation*}
            see Definition 1.19 of \cite{Lunardi}.
        \end{enumerate}

        \item We say that $\mathcal{E}$ belongs to the class $J_\theta=J_\theta(\mathcal{X}_0,\mathcal{X}_1)$ between $\mathcal{X}_0$ and $\mathcal{X}_1$ if  one of the following equivalent conditions holds: \begin{enumerate}[(i)]

            \item There exists $c>0$ such that \begin{equation*}
                \Vert l\Vert _{\mathcal{E}}\leq c\,\Vert l\Vert ^{1-\theta}_{\mathcal{X}_0}\Vert l\Vert ^\theta_{\mathcal{X}_1}, \quad \forall l\in \mathcal{X}_0\cap\mathcal{X}_1,
            \end{equation*}
           see Definition 1.19 of \cite{Lunardi}; 

            \item There exists $c>0$ such that \begin{equation*}
                \Vert l\Vert _{\mathcal{E}}\leq c\,t^{-\theta}J(t,l), \quad \forall l\in \mathcal{X}_0\cap\mathcal{X}_1,
            \end{equation*}
           see Lemma 1.10.1 of \cite{Triebel}. 
        \end{enumerate}
    \end{itemize}
\end{Def}
The proof of equivalence can be found in Lemma 1.10.1 of \cite{Triebel}.

\begin{Th}[Reiteration Theorem]\label{thm_reiteration}
    Let $0\leq \theta_0<\theta_1\leq 1$, and $0<\theta<1$. 
   If $\mathcal{E}_i$ belongs to $K_{\theta_i}\cap J_{\theta_i}$ ($i=0,1$) between $\mathcal{X}_0$ and $\mathcal{X}_1$, then we have \begin{equation*}
       (\mathcal{E}_0,\mathcal{E}_1)_\theta\cong (\mathcal{X}_0,\mathcal{X}_1)_{(1-\theta)\theta_0+\theta\theta_1}.
   \end{equation*}
\end{Th}
If $\mathcal{E}_i\in K_{\theta_i}$, Definition~\ref{def_K_J_classes} item (2) implies that $(\mathcal{E}_0,\mathcal{E}_1)_\theta$ is isomorphic to a subspace of $(\mathcal{X}_0,\mathcal{X}_1)_{(1-\theta)\theta_0+\theta\theta_1}$. And if $\mathcal{E}_i\in J_{\theta_i}$, the other side of the inclusion follows from Definition~\ref{def_K_J_classes} item (ii), we refer the readers to Theorem 1.10.2 of \cite{Triebel} and Theorem 1.23 of \cite{Lunardi} for further details.

Finally, since the continuous interpolation spaces $(\mathcal{X}_0,\mathcal{X}_1)_{\theta_i}$ ($i=0,1$) satisfy the conditions of Definition~\ref{def_K_J_classes} items (1) and (i) (by Lemma~\ref{lemma_interpolation_inequality}), when combining them with the result from Theorem~\ref{thm_reiteration}, we conclude that \begin{equation}\label{equ_reiteration}
    \left( (\mathcal{X}_0,\mathcal{X}_1)_{\theta_0}, (\mathcal{X}_0,\mathcal{X}_1)_{\theta_1}\right)_\theta \cong (\mathcal{X}_0,\mathcal{X}_1)_{(1-\theta)\theta_0+\theta\theta_1}.
\end{equation}

\section{Tools and outline of the proof}\label{section_Simonett}
In this section, we outline the proof of the main theorem using interpolation theory. We begin in Section~\ref{subsection_simonett} by presenting Simonett's stability theorem for quasilinear parabolic equations. In Section~\ref{subsection_obstructions}, we discuss its applications to compact manifolds as well as the difficulties encountered in the case of cusped manifolds. Finally, Section~\ref{subsection_outline} introduces a new proof strategy based on Angenent's maximal regularity result for linear equations. We also provide an overview of the structure of the remainder of the paper.

\subsection{Simonett's theorem}\label{subsection_simonett}
\begin{Th}[Simonett, Theorem 5.8 of \cite{Simonett}]\label{Simonett}
    Let $\mathcal{X}_1\hookrightarrow\mathcal{X}_0$ and $\mathcal{E}_1\hookrightarrow\mathcal{E}_0$ be continuous dense inclusions of Banach spaces. For fixed $0<\beta<\alpha<1$, let $\mathcal{X}_\alpha$ and $\mathcal{X}_\beta$, also denoted by $(\mathcal{X}_0,\mathcal{X}_1)_\alpha$ and $(\mathcal{X}_0,\mathcal{X}_1)_\beta$, respectively, be the continuous interpolation spaces corresponding to the inclusion $\mathcal{X}_1\hookrightarrow\mathcal{X}_0$. Let \begin{equation}\label{simonnet_equ}
        \frac{\partial}{\partial t}h(t)=\mathcal{A}(h(t))h(t)
    \end{equation}
    be an autonomous quasilinear parabolic equation for all $t\geq 0$, such that $\mathcal{A}(\cdot)\in C^k(\mathcal{G}_\beta,\mathcal{L}(\mathcal{X}_1,\mathcal{X}_0))$ for some positive integer $k$ and some open set $\mathcal{G}_\beta\subset \mathcal{X}_\beta$, where $\mathcal{L}(X,Y)$ represents the spaces of bounded linear operators from $X$ to $Y$. 
    
    Moreover, assume the following conditions hold. 
\begin{enumerate}[(C1)]
        \item For each $h\in \mathcal{G}_\beta$, the domain $D(\mathcal{A}(h))$ contains $\mathcal{X}_1$. Additionally, there exists an extension $\Tilde{\mathcal{A}}(h)$ of $\mathcal{A}(h)$ to a domain $D(\Tilde{\mathcal{A}}(h))$ that contains $\mathcal{E}_1$.
\end{enumerate}

Let $\mathcal{G}_\alpha:=\mathcal{G}_\beta\cap\mathcal{X}_\alpha$, the following conditions (C2)-(C4) hold for each $h\in\mathcal{G}_\alpha$.
\begin{enumerate}[resume*]
        \item 
        $\mathcal{A}(h)$ agrees with the restriction of $\Tilde{\mathcal{A}}(h)$ to the dense subset $D(\mathcal{A}(h))$ of $\mathcal{X}_0$.

        \item 
        $\Tilde{\mathcal{A}}(h)\in\mathcal{L}(\mathcal{E}_1,\mathcal{E}_0)$ generates a strongly continuous analytic semigroup on $\mathcal{L}(\mathcal{E}_1,\mathcal{E}_0)$.

        \item There exists $\theta\in (0,1)$, such that the following statement is true. Denote by $\left(\mathcal{E}_0, D(\Tilde{\mathcal{A}}(h))\right)_\theta$ the continuous interpolation space. And define the following set \begin{equation*}
            \left(\mathcal{E}_0, D(\Tilde{\mathcal{A}}(h))\right)_{1+\theta}:=\left\{l\in D(\Tilde{\mathcal{A}}(h)): \Tilde{\mathcal{A}}(h)(l)\in (\mathcal{E}_0, D(\Tilde{\mathcal{A}}(h)))_\theta\right\},
        \end{equation*}
        endowed with the graph norm of $\Tilde{\mathcal{A}}(h)$ with respect to $(\mathcal{E}_0, D(\Tilde{\mathcal{A}}(h)))_\theta$. Then there exists $\theta\in (0,1)$, such that \begin{equation*}
            \mathcal{X}_0\cong \left(\mathcal{E}_0, D(\Tilde{\mathcal{A}}(h))\right)_\theta,\quad \mathcal{X}_1\cong \left(\mathcal{E}_0, D(\Tilde{\mathcal{A}}(h))\right)_{1+\theta}.
        \end{equation*}

        \item $\mathcal{E}_1\hookrightarrow\mathcal{X}_\beta\hookrightarrow\mathcal{E}_0$ is a continuous and dense inclusion satisfying the following. There exist $c>0$ and $\delta\in (0,1)$ such that all $l\in \mathcal{E}_1$ has the property \begin{equation*}
            \Vert l\Vert _{\mathcal{X}_\beta}\leq c\Vert l\Vert ^{1-\delta}_{\mathcal{E}_0}\Vert l\Vert ^\delta_{\mathcal{E}_1}.
        \end{equation*}

    \end{enumerate}

    Let $h_0\in\mathcal{G}_\alpha$ be a fixed point of equation \eqref{simonnet_equ}. Suppose that the spectrum of the linearized operator $A_{h_0}:=D\mathcal{A}(h)|_{h=h_0}$ is contained in $\{z\in\mathbb{C}: \Re{z}\leq -\omega_0\}$ for some positive number $\omega_0$. 
    Then for any $\omega \in (0,\omega_0)$, there exist $\rho_0, C>0$, such that
    \begin{equation*}
        \Vert h(t)-h_0\Vert _{\mathcal{X}_1}\leq \frac{C}{t^{1-\alpha}}e^{-\omega t}\Vert h(0)-h_0\Vert _{\mathcal{X}_\alpha},\quad \forall t> 0,
    \end{equation*}
    for all solutions $h(t)$ of equation \eqref{simonnet_equ} with $h(0)\in B_{\mathcal{X}_\alpha}(h_0, \rho_0)$, the open ball of radius $\rho_0$ centered at $h_0$ in $\mathcal{X}_\alpha$. 
    
\end{Th}
Note that when $t=0$, we only need $h(0)-h_0$ to be contained in $\mathcal{X}_\alpha$, while for any positive time, $h(t)-h_0$ belongs to a smaller space $\mathcal{X}_1$, indicating that the solutions become more regular over time compared to the initial values.

\subsection{Obstructions in finite-volume manifolds}\label{subsection_obstructions}

Consider the normalized Ricci-DeTurck flow, and let the operator $\mathcal{A}$ in \eqref{simonnet_equ} to be the DeTurck operator, which is the expression on the right-hand side of the normalized Ricci-DeTurck flow. 

To determine whether Simonett's theorem applies to the Ricci-DeTurck flow, the first obstacle is to identify suitable little H{\"o}lder spaces. On compact manifolds $M$, Guenther, Isenberg, and Knopf \cite{Guenther-Isenberg-Knopf}, as well as Knopf and Young \cite{Knopf-Young}, choose the Banach spaces $\mathcal{X}_i$, $\mathcal{E}_i$ for $i=1,2$ in Theorem~\ref{Simonett} to be little H{\"o}lder spaces, defined as the closure of $C_c^\infty$ symmetric covariant 2-tensors compactly supported in $M$ with respect to the H{\"o}lder norms. More applications of Simonett's theorem can be found in \cite{Knopf} and \cite{Wu}.

However, when $M$ is a cusped hyperbolic manifold, the standard little H{\"o}lder spaces fail to satisfy condition (C3). In particular, due to the presence of Einstein variations, the operator $\mathcal{A}(h):\mathcal{X}_1\rightarrow\mathcal{X}_0$ is no longer surjective. For a detailed explanation and counterexamples, see Remark~\ref{rem_isom}. This motivates the introduction of a weight to the little H{\"o}lder spaces to restore surjectivity. 

As discussed in Remark~\ref{rem_isom}, the only viable weight is one that enlarges the domain $\mathcal{X}_1$ to allow tensors that grow exponentially toward the cusp. However, this introduces a second obstruction: when $h$ becomes unbounded in $C^2$, the operator $\mathcal{A}(h)$ is no longer bounded or necessarily well-defined on the new space. 
Although we can define $\mathcal{A}(h)$ as the DeTurck operator when $h$ is sufficiently close to $h_0$ in $C^2$ and then extend it as a bounded linear operator, this extension fails to be $C^1$ at points corresponding to blowing-up tensors. The $C^1$ regularity is crucial for establishing the existence and uniqueness of the solution (the fixed point argument requires $\mathcal{A}$ to be at least Lipschitz continuous, which also fails in this setting) and for deriving attractivity estimate. Therefore, Theorem~\ref{Simonett} does not readily apply in the case of cusped manifolds.

Despite these limitations, given any initial metric $h(0)$ that is $C^0$ close to $h_0$, the Ricci flow theory guarantees existence and uniqueness of the solution $h(t)$. Moreover, the stability theorem (Theorem~\ref{thm:Ckstability}) shows (up to taking $h$ closer to $h_0$) that $h(t)$ remains in a fixed $C^k$ neighborhood of $h_0$ for all $t\geq 1$. As a result, the linearization of the DeTurck operator at $h(t)$, denoted $A_{h(t)}=D\mathcal{A}(h)|_{h=h(t)}$, extends to a bounded linear operator. In contrast to Simonett's approach, we apply the linear theory for $A_{h(t)}$ to analyze the regularity and asymptotic behavior of the solution, see Theorem~\ref{thm_linear} below.

\subsection{Outline of the proof}\label{subsection_outline}
In Section~\ref{subsection_weighted spaces}, we introduce the weighted norms and weighted little H{\"o}lder spaces.
We then discuss how to define the linear operator $A_h$ at metrics that are $C^2$ close to $h_0$, as detailed in Section~\ref{subsection_DA}.

Furthermore, for a fixed $\alpha\in (0,1)$, define \begin{align}\label{equ_C_alpha}
    C_\alpha^0\left((0,\infty),\mathcal{X}_0\right):=&\left\{F\in C^0((0,\infty),\mathcal{X}_0): \lim_{t\rightarrow 0}t^{1-\alpha}\Vert F(t)\Vert _{\mathcal{X}_0}=0\right\},\\\nonumber
    C_\alpha^1\left((0,\infty),\mathcal{X}_0,\mathcal{X}_1\right):=&\Big\{g\in C^1\left((0,\infty),\mathcal{X}_0\right)\cap C^0\left((0,\infty),\mathcal{X}_1\right):\\\nonumber
    &\lim_{t\rightarrow 0}t^{1-\alpha}\left(\Vert g'(t)\Vert _{\mathcal{X}_0}+\Vert g(t)\Vert_{\mathcal{X}_1}\right)=0\Big\}.
\end{align}
Consider the linear problem \begin{equation}\label{equ_linear_A}
    \frac{\partial}{\partial t}g(t)=A g(t)+F(t),
\end{equation}
with initial data $g(0)$. The map $g(t)\mapsto g(0)$ is denoted by $I_\alpha$. 
Let \begin{align*}
    \mathcal{H}(\mathcal{X}_1,\mathcal{X}_0):=&\big\{A\in \mathcal{L}(\mathcal{X}_1,\mathcal{X}_0):\\
    &A\text{ generates a strongly continuous analytic semigroup}\big\},\\
    \mathcal{M}_\alpha(\mathcal{X}_1,\mathcal{X}_0):=& \Big\{A\in \mathcal{H}(\mathcal{X}_1,\mathcal{X}_0): \\
    &(\partial_t-A,I_\alpha)\in Isom \left(C_\alpha^1((0,\infty),\mathcal{X}_0,\mathcal{X}_1),C_\alpha^0((0,\infty),\mathcal{X}_0)\times \mathcal{X}_\alpha\right)\Big\}.
\end{align*}
In other words, $\mathcal{M}_\alpha(\mathcal{X}_1,\mathcal{X}_0)\subset \mathcal{H}(\mathcal{X}_1,\mathcal{X}_0)$ consists of the operators for which the differential equation \eqref{equ_linear_A} admits a unique solution $g(t)\in C_\alpha^1\left((0,\infty),\mathcal{X}_0,\mathcal{X}_1\right)$ for any given pair $(F,g(0))\in C_\alpha^0\left((0,\infty),\mathcal{X}_0\right)\times \mathcal{X}_\alpha$. 

Building on the linear theory above, Section~\ref{section_proof} presents the proof of the main theorem.
Suppose that $A_{h_0}\in \mathcal{M}_\alpha(\mathcal{X}_1,\mathcal{X}_0)$,
the stability theorem for the Ricci flow then allows us to express the solution as \begin{equation*}
    h(t)= e^{tA_{h_0}} h(0)+\int_0^t e^{(t-s)A_{h_0}}(\mathcal{A}(h(s))-A_{h_0})h(s)\,ds,
\end{equation*}
where $\mathcal{A}(h)$ will denote the DeTurck operator, $A_{h_0}$ (which will take the role of $A$ in \eqref{equ_linear_A}) the linearization of the normalized Ricci DeTurck flow at the fixed hyperbolic metric and $(\mathcal{A}(h(t))-A_{h_0})h(t)$ will take the role of $F(t)$ in \eqref{equ_linear_A}. We use this representation to establish the attractivity statement in Theorem~\ref{thm_main}.

Thus, it remains to verify $A_{h_0}\in\mathcal{M}_\alpha(\mathcal{X}_1,\mathcal{X}_0)$. Applying the following theorem, the problem reduces to checking that $A_{h_0}$ satisfies conditions \ref{(C1)}-\ref{(C4)}, which are verified in Section~\ref{section_semigroup generator}.

\begin{Th}[Angenent, Theorem 2.1.4 of \cite{Angenent_1990}]\label{thm_linear}
    Let $\mathcal{X}_1\hookrightarrow\mathcal{X}_0$ and $\mathcal{E}_1\hookrightarrow\mathcal{E}_0$ be continuous dense inclusions of Banach spaces.
    Let $A\in \mathcal{L}(\mathcal{X}_1,\mathcal{X}_0)$ be a linear operator. Assume that the following conditions hold. 
\begin{enumerate}[label=(C\arabic*)]
        \item \label{(C1)} The domain $D(A)$ contains $\mathcal{X}_1$. Additionally, there exists an extension $\Tilde{A}$ of $A$ to a domain $D(\Tilde{A})$ that contains $\mathcal{E}_1$.

        \item \label{(C2)} $A$ agrees with the restriction of $\Tilde{A}$ to the dense subset $D(A)$ of $\mathcal{X}_0$.
        \item \label{(C3)} $\Tilde{A}\in\mathcal{L}(\mathcal{E}_1,\mathcal{E}_0)$ generates a strongly continuous analytic semigroup on $\mathcal{L}(\mathcal{E}_1,\mathcal{E}_0)$, that is, $\Tilde{A}\in \mathcal{H}(\mathcal{E}_1,\mathcal{E}_0)$.
        \item \label{(C4)} There exists $\theta\in (0,1)$, such that the following statement is true. Denote by $(\mathcal{E}_0, D(\Tilde{A}))_\theta$ the continuous interpolation space. And define the following set \begin{equation}\label{def_1+theta}
            (\mathcal{E}_0, D(\Tilde{A}))_{1+\theta}:=\left\{l\in D(\Tilde{A}): \Tilde{A}(l)\in (\mathcal{E}_0, D(\Tilde{A}))_\theta\right\},
        \end{equation}
        endowed with the graph norm of $\Tilde{A}$ with respect to $(\mathcal{E}_0, D(\Tilde{A}))_\theta$. Then there exists $\theta\in (0,1)$, such that \begin{equation}\label{equ_isom}
            \mathcal{X}_0\cong (\mathcal{E}_0, D(\Tilde{A}))_\theta,\quad \mathcal{X}_1\cong (\mathcal{E}_0, D(\Tilde{A}))_{1+\theta}.
        \end{equation} 
    \end{enumerate}
Then $A\in\mathcal{M}_\alpha(\mathcal{X}_1,\mathcal{X}_0)$ for each $\alpha\in (0,1)$.
\end{Th}

\section{Weighted little H{\"o}lder spaces}\label{section_weight}

In this section, we introduce weighted little H{\"o}lder spaces. Specifically, condition \ref{(C3)} requires the operator $\omega I-A_{h_0}$ to be an isomorphism between the relevant Banach spaces for all $\omega$ greater than some fixed constant. This, in turn, necessitates introducing an additional exponential weight in the thin part of the cusps.

A similar approach was employed by Wu \cite{Wu}, who studied the stability of normalized Ricci flow on complex hyperbolic spaces $\mathbb{CH}^n$. In contrast to our setting, the analysis in Wu's work requires incorporating a weight function to account for the infinite volume of $\mathbb{CH}^n$. Those weighted little H{\"o}lder spaces are defined using an atlas covering $\mathbb{CH}^n$ that consists of a central disk and a sequence of overlapping annuli, with a weight on each annulus determined inductively.

\subsection{Weighted norms and little H{\"o}lder spaces}\label{subsection_weighted spaces}
To start our discussion, let $s>0$. For each $x\in M$, let $\Tilde{B}(x)\subset \mathbb{H}^3$ be the unit ball centered at a lift of $x$. For each tensor $l$ on $M$, the lift of $l$ on $\mathbb{H}^3$ is still denoted by $l$.

We define the following weighted little H{\"o}lder spaces on $M$.

\begin{Def}[weighted little H{\"o}lder spaces]\label{def_little_holder}

    Given $\lambda\in (0,1]$, $s\geq 0$ and $0<\alpha<1$, the \emph{weighted H{\"o}lder norm} $\Vert \cdot\Vert _{\mathfrak{h}^{k+\alpha}_{\lambda,s}}$ is defined as

\begin{align}\label{equ_weighted norm}
    \Vert l\Vert _{\mathfrak{h}^{k+\alpha}_{\lambda,s}}: &= \sup_{x\in M} \ww_\lambda(x)\Vert l|_{\Tilde{B}(x)} \Vert_{\mathfrak{h}^{k+\alpha}}\\\nonumber
    &=\sup_{x\in M, 0\leq j\leq k} \left( \ww_\lambda(x)|\nabla^j\,l(x)| + \sup_{y_1\neq y_2\in \Tilde{B}(x)} \ww_\lambda(x)\frac{|\nabla^kl(y_1)-\nabla^kl(y_2)|}{d_{\Tilde{B}(x)}(y_1,y_2)^\alpha} \right)
\end{align}
where 
\begin{align*}
\ww_\lambda(x) &=\begin{cases}
    e^{-\lambda r(x)}\quad \lambda\in (0,1),\\
    (r(x)+1)e^{-r(x)}\quad \lambda=1.
\end{cases}
\end{align*}
and
\begin{align*}
r(x) &=\begin{cases}
    0\quad\text{ if }x\in M(s),\\
    \text{dist}(x, \partial M(s))=\min_k (\text{dist}(x,T_k\times \{s\})\quad\text{otherwise}.
\end{cases}\\
\end{align*}
The $(r+1)$ multiplicative factor for $\ww_1$ is so that
\[\Vert l \Vert_{L^2(M)} \leq C_{\lambda,s} \Vert l\Vert _{\mathfrak{h}^{k+\alpha}_{\lambda,s}},
\]
holds.

As for fixed $\lambda$ the function $\ww_\lambda(x)$ satisfies
\[|\nabla^j \ww_\lambda(x)| \leq C_j \ww_\lambda(x)
\]
we can easily check that the norm $\Vert l \Vert_{\mathfrak{h}^{k+\alpha}_{\lambda,s}}$ is equivalent to
\[\sup_{x\in M, 0\leq j\leq k} \left( |\nabla^j(\ww_\lambda(x)\,l(x))| + \sup_{y_1\neq y_2\in \Tilde{B}(x)} \ww_\lambda(x)\frac{|\nabla^kl(y_1)-\nabla^kl(y_2)|}{d_{\Tilde{B}(x)}(y_1,y_2)^\alpha} \right)
\]

The \emph{little H{\"o}lder space} $\mathfrak{h}^{k+\alpha}_{\lambda,s}$ is defined to be the closure of $C_c^\infty$ symmetric covariant 2-tensors compactly supported in $M$ with respect to the weighted H{\"o}lder norm $\Vert \cdot\Vert _{\mathfrak{h}^{k+\alpha}_{\lambda,s}}$. Analogously by only considering the $C^k$ norm instead of $\mathfrak{h}^{k+\alpha}$ we define the spaces $C^k_{\lambda,s}$ with their corresponding norm.

Moreover, for fixed $0<\sigma<\rho<1$, we define \begin{align}\label{def_X_i,E_i}
    \mathcal{X}_{0}=\mathcal{X}_{0}(M,\rho,\lambda,s)=:\mathfrak{h}^{0+\rho}_{\lambda,s},&\quad \mathcal{X}_{1}=\mathcal{X}_{1}(M,\rho,\lambda,s)=:\mathfrak{h}^{2+\rho}_{\lambda,s}\\\nonumber
    \mathcal{E}_{0}=\mathcal{E}_{0}(M,\sigma,\lambda,s)=:\mathfrak{h}^{0+\sigma}_{\lambda,s},&\quad \mathcal{E}_{1}=\mathcal{E}_{1}(M,\sigma,\lambda,s)=:\mathfrak{h}^{2+\sigma}_{\lambda,s}
\end{align}
Observe that \begin{equation*}
    \mathcal{X}_{1}\subset \mathcal{E}_{1}\subset \mathcal{X}_{0}\subset \mathcal{E}_{0}.
\end{equation*}
\end{Def}

Analogously to the results for standard little H{\"o}lder spaces $\mathfrak{h}^{k+\alpha}$, we have the following properties for the weighted spaces.
\begin{Prop}\label{prop_interp_isom}
    For any $\theta\in (0,1)$ and $m\in\mathbb{N}$ with $m\theta\notin\mathbb{N}$, \begin{equation*}
    (C_{\lambda,s}^0, C_{\lambda,s}^m)_\theta\cong \mathfrak{h}_{\lambda,s}^{m\theta},
\end{equation*}
where the \emph{weighted $C^k$-space} for $k\in\mathbb{N}$ is defined as: 
\[\Vert l\Vert _{C^{k}_{\lambda,s}(M)} : = \sup_{x\in M, 0\leq j\leq k} \left( e^{-\lambda r(x)}|\nabla^j\,l(x)| \right)
\]
\end{Prop}
The proof of the proposition is presented in Appendix~\ref{appendix_interpolation}.

\begin{Cor}\label{cor_reiteration}
    Given $\theta\in (0,1)$ with $2\theta+\rho, 2\theta+\sigma\notin\mathbb{N}$, we have \begin{equation*}
        (\mathcal{X}_{0},\mathcal{X}_{1})_\theta \cong \mathfrak{h}_{\lambda,s}^{2\theta+\rho},\quad (\mathcal{E}_{0},\mathcal{E}_{1})_\theta \cong \mathfrak{h}_{\lambda,s}^{2\theta+\sigma}.
    \end{equation*}
\end{Cor}

\begin{proof}
We only provide the proof for the first isomorphism, and the second follows for the same reason.
Proposition~\ref{prop_interp_isom} implies the existence of $\theta_0, \theta_1\in (0,1)$ and $m\in\mathbb{N}$, such that \begin{equation*}
   (C_{\lambda,s}^0, C_{\lambda,s}^m)_{\theta_0}\cong \mathfrak{h}_{\lambda,s}^{m\theta_0} = \mathfrak{h}_{\lambda,s}^{0+\rho},\quad (C_{\lambda,s}^0, C_{\lambda,s}^m)_{\theta_1}\cong \mathfrak{h}_{\lambda,s}^{m\theta_1} = \mathfrak{h}_{\lambda,s}^{2+\rho}.
\end{equation*}
It indicates that \begin{equation*}
   m\theta_0=\rho,\quad m\theta_1=2+\rho \quad \Longrightarrow \quad m\left((1-\theta)\theta_0+\theta\theta_1\right)=2\theta +\rho.
\end{equation*}
Applying the Reiteration Theorem~\ref{thm_reiteration} and equation \eqref{equ_reiteration}, we have
\begin{align*}
     (\mathcal{X}_{0},\mathcal{X}_{1})_\theta = & (\mathfrak{h}_{\lambda,s}^{0+\rho}, \mathfrak{h}_{\lambda,s}^{2+\rho})_\theta \cong \left((C_{\lambda,s}^0, C_{\lambda,s}^m)_{\theta_0}, (C_{\lambda,s}^0, C_{\lambda,s}^m)_{\theta_1}\right)_\theta \\
     \cong & (C_{\lambda,s}^0, C_{\lambda,s}^m)_{(1-\theta)\theta_0+\theta\theta_1}
     \cong  \mathfrak{h}_{\lambda,s}^{m((1-\theta)\theta_0+\theta\theta_1)} \\
     = & \mathfrak{h}_{\lambda,s}^{2\theta +\rho}.
\end{align*}
\end{proof}

To conclude this section, we explain the reasoning behind introducing the exponential weight in \eqref{equ_weighted norm}.
\begin{Rem}\label{rem_isom}
  To achieve the exponential decay of the solution to Ricci flow toward the hyperbolic metric, we want the real spectrum of the operator $A_{h_0}$ to be bounded above by a negative number $\omega_0$.
  For any $\omega$ with $\Re\omega> \omega_0$, we need to confirm that the operator $\omega I-A_{h_0}$ acting between unweighted H{\"o}lder spaces $\mathfrak{h}^{2+\rho}(M)$ and $ \mathfrak{h}^{0+\rho}(M)$ is an isomorphism. 
  This holds true for compact hyperbolic manifolds.
        
  However, for cusped manifolds, the main obstruction is that, for any real number $\omega<0$, the map $\omega I-A_{h_0}: \mathfrak{h}^{2+\rho}(M) \rightarrow \mathfrak{h}^{0+\rho}(M)$ is no longer surjective. Let $f=(\omega-A_{h_0})(l)\in \mathfrak{h}^{0+\rho}(M)$. Analogous to the ODE estimates in Lemma~\ref{claim}, calculations show that, $\hat{l}$, the average of $l$ on $T\times \{r\}$ defined in \eqref{equ_hat_l} satisfies \begin{equation*}
    e^{2r}\hat{l}_{ij})''-2(e^{2r}\hat{l}_{ij})'-\omega e^{2r}\hat{l}_{ij} = O\left(\Vert f\Vert _{\mathfrak{h}^{0+\rho}(M)}\right),
  \end{equation*}
    where $1\leq i,j\leq 2$. The characteristic polynomial has roots equal to $1\pm\sqrt{1+\omega}$. When $\omega<0$, $1-\sqrt{1-\omega}>0$. It means that $\hat{l}$ (and therefore $l$) may diverge as $r\rightarrow\infty$, the solution $l$ may not belong to $\mathfrak{h}^{2+\rho}(M)$. 
        
        Therefore, we need to adjust the spaces by shrinking the target space $\mathfrak{h}^{0+\rho}(M)$, or enlarging the domain $\mathfrak{h}^{2+\rho}(M)$, or applying both adjustments. Additionally, an isomorphism as demonstrated in Corollary~\ref{cor_reiteration} is required. This was the motivation for the exponential weight in the little H{\"o}lder spaces.
        If the weight of $\mathcal{X}_{0}$ and $\mathcal{X}_{1}$ were of the form $e^{\lambda r(x)}$ with $\lambda>0$, the situation would be worse.
        Consider the \emph{Einstein variation} $u$ on a cusp, which is a $(0,2)$-tensor of the form \begin{equation*}
            u=e^{-2r}u_{ij}dx^idx^j,
        \end{equation*}
        where $u_{ij}\in\mathbb{R}$. The variation $u$ is called a \emph{trivial Einstein variation} if its trace vanishes everywhere with respect to the flat metric on the torus. 
        Notice that the operator $A_{h_0}$, when restricted to the cusp, acts on $u$ to produce zero.
        Let $\rho$ be a cutoff function supported in a neighborhood of the cusp, taking the value one on $T\times [0,\infty)$, and let $l\in \mathcal{X}_{1}$ be a tensor with compact support. Then, defining $f:=-A_{h_0}(l+\rho u)$, we find that $f$ is compactly supported, so $f\in \mathcal{X}_{0}$. However, the preimage $l+\rho u$ does not decay, meaning it does not belong to $\mathcal{X}_{1}$. 

        Based on this reasoning, we introduce a weight of $e^{-\lambda r(x)}$ instead, where $\lambda\in (0,1)$. As stated, when $\lambda=1$, the additional factor of $(r+1)$ ensures that the little H{\"o}lder spaces are contained in $L^2(M)$. This ensures that $\omega I-A_{h_0}:\mathcal{X}_{1}\rightarrow\mathcal{X}_{0}$ is bijective for all $\omega$ whose real part is larger than a negative constant. 
\end{Rem}

\subsection{Extension of the linearization}\label{subsection_DA}
For any $h\in C^2(M)$, let $\mathcal{A}(h)$ be the DeTurck operator, given by the expression on the right-hand side of \eqref{DRF}. For each $l\in \mathcal{X}_1\cap C^2(M)$, by Proposition 2.3.7 of \cite{Topping}, the linearization \begin{equation*}
    A_{h}(l) = \lim_{\xi\rightarrow 0}\frac{\mathcal{A}(h+\xi l)(h+\xi l)-\mathcal{A}(h)(h)}{\xi}=\Delta_Ll+\mathcal{L}_{(\delta G(l))^{\#}}h-4l \in \mathcal{X}_0,
\end{equation*}
where $G(l)=l-\frac{1}{2}(\text{tr}l)h$.

In general, for each $l\in \mathcal{X}_1$, from the definition of the little H{\"o}lder space $\mathcal{X}_1$, there exists a sequence of smooth, compactly supported tensors $l_n\in \mathcal{X}_1$, such that $l_n$ converges to $l$ in the $\mathcal{X}_1$-norm. We define \begin{equation*}
    A_h(l):=\lim_{n\rightarrow\infty} A_h(l_n)\in\mathcal{X}_0.
\end{equation*}
Then we have $A_h\in \mathcal{L}(\mathcal{X}_1,\mathcal{X}_0)$ for each $h\in C^2(M)$.

\section{Generators of Analytic Semigroups}\label{section_semigroup generator}
In this section, we prove that $A_{h_0}\in\mathcal{M}_\alpha(\mathcal{X}_1,\mathcal{X}_0)$ for each $\alpha \in (0,1)$. 
According to Theorem~\ref{thm_linear}, it suffices to verify conditions \ref{(C1)}-\ref{(C4)} for $A_{h_0}$.
\subsection{Conditions \ref{(C1)}-\ref{(C2)}}
\begin{enumerate}[(C1)]
        \item Let $\Tilde{A}_{h_0}$ be the extension of $A_{h_0}$ that maps from $\mathcal{E}_{1}$ to $\mathcal{E}_{0}$, where the domain $D(\Tilde{A}_{h_0})=\mathcal{E}_{1}$ is dense in $\mathcal{E}_{0}$. 

        \item By construction in (C1), $A_{h_0}$ agrees with the restriction of $\Tilde{A}_{h_0}$ on $D(A_{h_0})=\mathcal{X}_{1}$. 
\end{enumerate}

\subsection{Condition \ref{(C3)}}
In this subsection, we demonstrate that $\Tilde{A}_{h_0}$ generates a strongly continuous analytic semigroup on $\mathcal{L}(\mathcal{E}_{1},\mathcal{E}_{0})$. 
        This result is comparable to Lemma 3.4 of \cite{Guenther-Isenberg-Knopf}. For a compact hyperbolic manifold $M$, the operator $-\Tilde{A}_{h_0}$ acts as an isomorphism between unweighted H{\"o}lder spaces $\mathfrak{h}^{2+\sigma}(M)$ and $ \mathfrak{h}^{0+\sigma}(M)$, and the same holds for $\omega I-\Tilde{A}_{h_0}$ for each $\omega>0$. 
        For cusped manifolds, as we discussed in Remark~\ref{rem_isom}, establishing the isomorphism property requires solving a system of ODEs on the cusps and verifying the surjectivity of $\omega I-\Tilde{A}_{h_0}$ for $\Re (\omega)>\omega_0$ and a suitable choice of $\omega_0>0$. 
        This strategy is inspired by the work of Bamler \cite{Bamler_Dehnfilling} and Hamenst{\"a}dt-J{\"a}ckel \cite{Hamenstadt-Jackel}, who observed that by averaging solutions of the linear equation $(\omega I-\Tilde{A}_{h_0})(l)=f$ over cross sections of the cusps, the problem reduces to a system of ODEs to then conclude the statement by Poincar\'e inequality. These ODEs provide precise control over the asymptotic behavior of solutions on the cusps and facilitate the construction of an isomorphism between the weighted spaces $\mathcal{E}_1$ and $\mathcal{E}_0$.
         
        Another important step is the derivation of $\Vert l\Vert _{\mathfrak{h}^{2+\sigma}}\lesssim \Vert (\omega I-\Tilde{A}_{h_0})(l)\Vert _{\mathfrak{h}^{0+\sigma}}$ using Schauder estimates for tensors on $M$, which depends on the lower bound of injectivity radius of the compact manifold in \cite{Guenther-Isenberg-Knopf}. However, for cusped $M$, due to the lack of a positive lower bound on the injectivity radius of $M$, we define the weighted H{\"o}lder norm by passing to the universal cover $\mathbb{H}^3$ which admits infinite injectivity radius,
        and will import the standard Schauder estimates from $\mathbb{R}^3$ using this norm. We then aim to prove the following result.

        \begin{Prop}\label{prop_C_3}
        There exists $\omega_0\in\mathbb{R}$ such that for any $\Re(\omega)> \omega_0$ the operator $\omega I-\Tilde{A}_{h_0}\in \mathcal{L}(\mathcal{E}_{1},\mathcal{E}_{0})$ is invertible, and the inverse operator $(\omega I-\Tilde{A}_{h_0})^{-1}\in \mathcal{L}(\mathcal{E}_{0},\mathcal{E}_{1})$.
        \end{Prop}

        In Section~\ref{step1} and \ref{step2}, we will show that for any $\omega\in\mathbb{C}$, there exists $C_\omega>0$, such that \begin{equation}\label{equ_c_3}
            \Vert l\Vert _{\mathcal{E}_{1}}\leq C_\omega \Vert (\omega I-\Tilde{A}_{h_0})(l)\Vert _{\mathcal{E}_{0}}\quad \forall l\in \mathcal{E}_{1}.
        \end{equation}
        This proves the injectivity. 
        In Section~\ref{step3}, we will find $\omega_0\in\mathbb{R}$, such that for any $\omega>\omega_0$, the map $\omega I-\Tilde{A}_{h_0}$ is also surjective. Therefore, the proposition follows from the bounded inverse theorem. 
    
        Furthermore, in Section~\ref{step4}, we prove the existence of a uniform constant $C=C(M,h_0,s,\omega_1)$ for $\omega_1>\omega_0$, such that for any $\Re(\omega) > \omega_1$, 
                \begin{equation*}
           \Vert l\Vert _{\mathcal{E}_{1}}\leq C \Vert (\omega I-\Tilde{A}_{h_0})(l)\Vert _{\mathcal{E}_{0}}\quad \forall l\in \mathcal{E}_{1}.
        \end{equation*}
        This estimate, together with the above proposition, provides a sufficient condition for \ref{(C3)} by Amann (\cite[Section 1.2]{Amann}). More details about the definition and properties of strongly continuous semigroups on Banach spaces can also be found in Chapter 2 of \cite{Lunardi_Analytic_Semigroups}.

        \subsubsection{Schauder estimates}\label{step1}
        
        Fix $\omega\in\mathbb{C}$.
        $(\omega I-\Tilde{A}_{h_0})(l)$ can be expressed as $-\Delta_{h_0}l$ plus lower order terms with bounded coefficients. 
        
        Consider the Schauder estimates applied to an operator which is expressed by $h^{pq}\partial^2_{pq}(l_{jk})$ plus lower order terms in local coordinates, where the coefficients of the lower order terms involve up to the second derivatives of $h$. 
        As discussed in the proof of Proposition 2.5 in \cite{Hamenstadt-Jackel}, in order to apply the Schauder estimates for tensors, we need to find a constant $r>0$ and a harmonic chart $\phi: B(x,r)\subset M\rightarrow \mathbb{R}^3$ for every $x\in M$, such that \begin{itemize}
            \item the radius $r>0$ is independent of $x$,
            \item the transformation matrix $(h_\phi^{jk})$ is uniformly elliptic,
            \item and $\Vert h^\phi_{jk}\Vert _{\mathfrak{h}^{2+\sigma}}$ admits a uniform upper bound. 
        \end{itemize}
        
        By page 230 of \cite{Anderson06}, the above conditions can be achieved if $h$ is a metric that possesses the following properties. 
        \begin{enumerate}[(i)]
            \item $\Vert Ric(h)\Vert _{C^1}\leq \Lambda$,
            \item $inj(M,h)\geq i_0$.
        \end{enumerate}
        In particular, the radius $r$ depends only on the H{\"o}lder exponent $\sigma$ in the definition of $\mathcal{E}_{0}$ and $\mathcal{E}_{1}$ \eqref{def_X_i,E_i}, as well as the given positive constants $\Lambda$ and $i_0$.

        Since we only consider $h=h_0$, condition (i) holds automatically. However, condition (ii) fails due to the absence of a positive lower bound on the injectivity radius for the noncompact manifold $M$.

        Nevertheless, the weighted H{\"o}lder norm is defined by passing to the universal cover with infinite injectivity radius.
        Recall the weight in \eqref{equ_weighted norm} denoted by $\ww_\lambda(x)$. Since the Lipschitz constant of $\ln\ww_\lambda(\cdot)$ is bounded by 1 on $M\setminus M(s)$, and $\ww_\lambda(\cdot)=1$ on $M(s)$, 
        the estimate with respect to $h_0$ can be viewed as contracting the norm near each point by a given exponential rate. 
        This serves as an alternative to condition (ii) and confirms the uniform ellipticity of the transformation matrix, thereby giving rise to the Schauder estimates for $\omega I-\Tilde{A}_{h_0}$ from the classical interior Schauder estimates on $\mathbb{R}^3$.
        As a consequence, we have        \begin{equation}\label{equ_Schauder}
            \Vert l\Vert _{\mathcal{E}_{1}}\leq C\left( \Vert (\omega I-\Tilde{A}_{h_0})(l)\Vert _{\mathcal{E}_{0}}+\Vert l\Vert _{C^0_{\lambda,s}(M)}\right)\quad \forall l\in \mathcal{E}_{1},
        \end{equation}
        where the constant $C$ depends only on $\omega$, $\sigma$, $\Lambda$, $i_0$, and the ellipticity of the second order term of $\Tilde{A}_{h_0}$.

\subsubsection{$C^0_{\lambda,s}$ estimates and injectivity}\label{step2}
Next, we aim to bound $\Vert l\Vert _{C^0_{\lambda,s}}$ by a uniform constant multiplied by $\Vert (\omega I-\Tilde{A}_{h_0})(l)\Vert _{\mathcal{E}_{0}}$.

 We start by considering the cusp region $\mathcal{C}_s=\cup_k T_k\times [s,\infty)$. Define the average tensor of $l$ on $\mathcal{C}_s$:
       \begin{equation}\label{equ_hat_l}
    \hat{l}_{ij}(x)=\frac{1}{\vol(T_k(r))}\int_{T_k(r)} l_{ij}(y)\dvol(y),
\end{equation}
where $x\in T_k(r):=T_k\times\{r\}$.

For functions $g:X\rightarrow\mathbb{C}$ and $h:X\rightarrow \mathbb{R}$, we write $g=O(h)$ if $|g(x)|\leq ch(x)$ for all $x\in X$.
We prove the following lemma. 

\begin{Lemma}\label{claim}
    Set
    \[
    \omega_0 = -\lambda(2-\lambda), \quad \lambda\in (0,1].
    \]
     For each $\Re(\omega)>\omega_0$ and $f\in\mathcal{E}_{0}$, suppose there is $l\in C^2(Sym^2(T^*M))\cap H^1(M)$ that solves \begin{equation*}
    (\omega I-\Tilde{A}_{h_0})(l)=f,
\end{equation*}
Then we have \begin{equation*}
    \Vert \hat{l}\Vert _{C^0_{\lambda,s}(\mathcal{C}_s)}=O(\Vert f\Vert _{\mathcal{E}_0}).
\end{equation*}
where constants depend on $\lambda, s$ and $\omega$.
\end{Lemma}

Observe that here we only assume that $l\in C^2(Sym^2(T^*M))$ is locally $C^2$, without requiring that $\Vert l\Vert _{C^2(M)}$ is finite. Additionally, since $C_c^\infty$ tensors on the noncompact manifold $M$ are dense in $H^1(M)$ (see, for instance, Theorem 2.4 of \cite{Hebey}), we have $H^1(M)=H_0^1(M)$.

     \begin{proof}
      Consider the average tensor of $f$ on the cusp:
\begin{equation*}
    (\hat{f})_{ij}(x)=\frac{1}{\vol(T_k(r))}\int_{T_k(r)} f_{ij}(y)\dvol(y),
\end{equation*}
where as before $x\in T_k(r):=T_k\times\{r\}$. Observe that in the cusp regions we have $(\omega I-\Tilde{A}_{h_0})(\hat{l})=\hat{f}$. Moreover, both $\hat{f}$ and $\hat{l}$ depend only on $r\in [s,\infty)$. As calculated in (9.14) of \cite{Hamenstadt-Jackel}, the equality $(\omega I-\Tilde{A}_{h_0})(\hat{l})=\hat{f}$ is equivalent to the following system of ODEs. 

\begin{align}\label{equ_odes}
    &(e^{2r}\hat{l}_{ij})''-2(e^{2r}\hat{l}_{ij})'-\omega e^{2r}\hat{l}_{ij} &=& -e^{2r}\hat{f}_{ij}+2\delta_{ij}(tr(\hat{l})-\hat{l}_{33}),\quad i, j=1,2,\\\nonumber
   & (e^r\hat{l}_{i3})''-2(e^r\hat{l}_{i3})'-(3+\omega)e^r\hat{l}_{i3} &=& -e^r\hat{f}_{i3}, \quad i=1,2,\\\nonumber
   &\hat{l}_{33}''-2\hat{l}_{33}'-(4+\omega) \hat{l}_{33} &=& -\hat{f}_{33},\\\nonumber
\end{align}
By adding $e^{2r}\hat{l}_{11}, e^{2r}\hat{l}_{22}$, and $\hat{l}_{33}$, we obtain \begin{equation}\label{equ_trl}
    (tr(\hat{l}))''-2(tr(\hat{l}))'-(4+\omega) tr(\hat{l})=-tr(\hat{f}).
\end{equation}
The roots of the characteristic polynomials of $e^{2r}\hat{l}_{ij}$, $e^r\hat{l}_{i3}$, $\hat{l}_{33}$, and $tr(\hat{l})$ (where $i,j=1,2$) are $1\pm\sqrt{1+\omega}$, $1\pm\sqrt{4+\omega}$, $1\pm \sqrt{5+\omega}$, and $1\pm \sqrt{5+\omega}$, respectively. Here the square roots are chosen so that their real part is non-negative, where any arbitrary choice is made in the purely imaginary case. 
Since \begin{equation}\label{equ_est_f}
    |\hat{f}(r)|=O(\Vert f\Vert _{\mathcal{E}_{0}}\ww^{-1}_\lambda),
\end{equation}
the solutions to the system \eqref{equ_odes} are as follows.
\begin{align}\label{eq:ODEsolutions}
    e^{2r}\hat{l}_{12} &= a_1 e^{(1+\sqrt{1+\omega})r}+a_2 e^{(1-\sqrt{1+\omega})r}+O(\Vert f\Vert _{\mathcal{E}_{0}}\ww^{-1}_\lambda),\\\nonumber
    e^r\hat{l}_{i3} &= b_1^i e^{(1+\sqrt{4+\omega})r}+ b_2^i e^{(1-\sqrt{4+\omega})r}+O(\Vert f\Vert _{\mathcal{E}_{0}}\ww^{-1}_\lambda),\quad i=1,2,\\\nonumber
    \hat{l}_{33} &= c_1e^{(1+ \sqrt{5+\omega})r}+c_2e^{(1- \sqrt{5+\omega})r} +O(\Vert f\Vert _{\mathcal{E}_{0}}\ww^{-1}_\lambda),\\\nonumber
    tr(\hat{l}) &= d_1e^{(1+ \sqrt{5+\omega})r}+d_2e^{(1- \sqrt{5+\omega})r} +O(\Vert f\Vert _{\mathcal{E}_{0}}\ww^{-1}_\lambda).
\end{align}

Observe that $\hat{l}$ is $L^2$-integrable, as by applying Cauchy-Schwartz we have that for $x\in T_k(r)$ 
\[\left( \hat{l}_{ij}(x) \right)^2 \leq \frac{\int_{T_k(r)} l_{ij}^2(y) \dvol(y)}{\int_{T_k(r)}\dvol(y)},
\]
hence it follows that
\[
\Vert\hat{l}\Vert_{L^2(\mathcal{C}_s)}^2 =\int_s^{+\infty} e^{-2r}\vert\hat{l}\vert^2 dr \leq \Vert l\Vert_{L^2(\mathcal{C}_s)}^2.
\]
Moreover, 
\[
\left\Vert\Vert f\Vert _{\mathcal{E}_{0}}\ww^{-1}_\lambda\right\Vert_{L^2(\mathcal{C}_s)}^2 =\int_s^{+\infty} e^{-2r}\Vert f\Vert _{\mathcal{E}_{0}}^2\ww^{-2}_\lambda dr \lesssim \Vert f\Vert _{\mathcal{E}_{0}}^2.
\]
Therefore, we have that
\begin{equation*}
    e^{-r}(e^{2r}\hat{l}_{12}),\, e^{-r}(e^r\hat{l}_{i3}),\, e^{-r}\hat{l}_{33},\,e^{-r}tr(\hat{l})\in L^2([0,\infty)).
\end{equation*}
Observe that any root with real part greater than or equal to $1$ is not square integrable. Therefore, we must have that $a_1=b_1^i=c_1=d_1=0$.

Additionally, observe that for any $\omega\in\mathbb{C}$ with
\[
\Re(\omega) \geq -\lambda(2-\lambda), \quad \lambda \in (0,1] 
\]
we have 
\begin{equation*}
   \Re(1-\sqrt{1+\omega}),\, \Re(1-\sqrt{4+\omega}),\,\Re(1-\sqrt{5+\omega})\leq \lambda.
\end{equation*}

\begin{Claim}
    \begin{equation*}
    a_2,b_2^i,c_2,d_2=O(\Vert f\Vert _{\mathcal{E}_{0}}).
\end{equation*}
\end{Claim}
\begin{proof}
Observe that by replacing $r=0$ in \eqref{eq:ODEsolutions} and using that $a_1=b_1^i=c_1=d_1=0$ together with the definition of $\hat{l}$, it is sufficient to show that $l(x)=O(\Vert f\Vert_{\mathcal{E}_0})$ for $x\in \partial\mathcal{C}_s$.

By Proposition~\ref{prop:weightedL2bound} for $\xi=\sqrt{\frac{1+\Re(\omega)}{2}}$ we have \begin{equation*}
    \int_M e^{-\xi r(x)}(|l|^2+|\nabla l|^2)(x)\dvol \lesssim \int_M e^{-\xi r(x)} |f|^2\dvol.
\end{equation*}
For each $\lambda\leq 1$, $2-2\lambda+\xi\geq \xi>0$. Therefore, 
\begin{align}\label{equ_L^2_l}
    \int_M e^{-\xi r(x)}(|l|^2+|\nabla l|^2)(x)\dvol\lesssim & \Vert f\Vert^2_{C^0(M(s))}+\int_s^\infty\int_{T(r)}e^{-\xi r}\Vert f\Vert^2 _{\mathcal{E}_{0}}e^{2\lambda r}\dvol\,dr\\\nonumber
    \lesssim & \Vert f\Vert^2 _{\mathcal{E}_{0}}+ \Vert f\Vert^2 _{\mathcal{E}_{0}} \int_s^\infty e^{-(2-2\lambda+\xi)r}dr\\\nonumber
    =& O(\Vert f\Vert^2 _{\mathcal{E}_{0}}).
\end{align}

Let $\Tilde{M}$ be the universal cover of $M$, and let $\Tilde{x}$ be a lift of $x$ in $\Tilde{M}$. Furthermore, let $\Tilde{f}:=(\omega I-\Tilde{A}_{h_0})(\Tilde{l})$, and define $L:=\Tilde{l}$ and $F:=\Tilde{f}$. It follows that $$ F=(\omega I-\Tilde{A}_{h_0})(L)=-\Delta L+Ric(L)+(4+\omega)L.$$
 Since $\Delta\left(| L|^2\right)=2\Re\langle\Delta L,\bar{L}\rangle+2|\nabla L|^2$, by Lemma 3.2 of \cite{Hamenstadt-Jackel}, we have \begin{align*}
        \frac12 \Delta\left(|L|^2\right)
        =& \Re\langle\Delta L,\bar{L}\rangle+|\nabla L|^2\\
        =&-\Re\langle F,\bar{L}\rangle+\Re\langle Ric(L),\bar{L}\rangle+(4+\Re(\omega))|L|^2+|\nabla L|^2\\
        \geq& -|F||L|-6|L|^2+2|tr_{h_0}(L)|^2+(4+\Re(\omega))|L|^2+|\nabla L|^2\\
        \geq & -|F||L|-(2-\Re(\omega))|L|^2+|\nabla L|^2.
    \end{align*}
    On the other hand, since $|\nabla(|L|)|\leq |\nabla L|$, \begin{equation*}
        \frac12 \Delta\left(| L|^2\right)= |L|\Delta(|L|)+\left|\nabla(|L|)\right|^2
        \leq |L|\Delta(|L|)+|\nabla L|^2.
    \end{equation*}
    Combining these two inequalities and assuming $L\neq0$, we obtain \begin{equation*}
        \Delta(|L|)+(2-\Re(\omega))|L|\geq -|F|,
    \end{equation*}
    this verifies the condition for the De Giorgi-Nash-Moser estimate (see Theorem 8.17 in \cite{Gilbarg-Trudinger} or Lemma 2.8 in \cite{Hamenstadt-Jackel}). This implies \begin{equation}\label{De Giorgi-Nash-Moser_0}
        |L|(\Tilde{x})\leq C\left(\Vert L\Vert _{L^2(B(\Tilde{x}))} +\Vert F\Vert _{L^2(B(\Tilde{x}))}\right),
    \end{equation}
     where $B(.)$ denotes the unit ball in $\mathbb{H}^3$, $C$ depending only on $\omega$. As \eqref{De Giorgi-Nash-Moser_0} is stable under $C^2$ convergence, we can extend the inequality to arbitrary $L$.
    Applying it to the scalar functions $|L|=|\Tilde{l}|$ and $|F|=|\Tilde{f}|$, we obtain the following inequality. \begin{equation}\label{De Giorgi-Nash-Moser_1}
       |l|(x)= |\Tilde{l}|(\Tilde{x})\lesssim \Vert \Tilde{l}\Vert _{L^2(B(\Tilde{x}))} +\Vert \Tilde{f}\Vert _{L^2(B(\Tilde{x}))}.
    \end{equation}

    As $x\in \partial\mathcal{C}_s$ one can verify that the number of lifts of $x$ in $B(\Tilde{x})$ is bounded by a constant $C(s)$ depending on $s$ (see for instance \cite[Corollary 7.7]{Hamenstadt-Jackel}). This leads to \begin{equation*}
        \int_{B(\Tilde{x})} |\Tilde{l}|^2\dvol \lesssim \int_{B(x)}e^{r(y)}|l|^2(y)\dvol\leq e^{(1+\xi)(r(x)+1)}\int_{B(x)}e^{-\xi r(y)}|l|^2(y)\dvol.
    \end{equation*}
    Since $x\in \partial\mathcal{C}_s$, we have $r(x)=0$.
    Then by \eqref{equ_L^2_l},
    \begin{equation}\label{equ_L2tildel}
        \int_{B(\Tilde{x})} |\Tilde{l}|^2\dvol \lesssim \int_{B(x)}e^{-\xi r(y)}|l|^2(y)\dvol= O(\Vert f\Vert^2 _{\mathcal{E}_{0}}).
    \end{equation}
   Similarly, for the second term in \eqref{De Giorgi-Nash-Moser_1}, \begin{equation}\label{equ_L2tildef}
         \int_{B(\Tilde{x})} |\Tilde{f}|^2\dvol \lesssim \int_{B(x)}e^{r(y)}|f|^2(y)\dvol\lesssim \int_{B(x)}e^{-\xi r(y)}|f|^2(y)\dvol=O(\Vert f\Vert^2 _{\mathcal{E}_{0}}).
    \end{equation}
Substituting \eqref{equ_L2tildel} and \eqref{equ_L2tildef} into \eqref{De Giorgi-Nash-Moser_1},
\begin{equation*}
    \Vert l\Vert_{C^0(\partial\mathcal{C}_s)}=O(\Vert f\Vert _{\mathcal{E}_{0}}).
\end{equation*}
Then we obtain 
\begin{equation*}
    |\hat{l}(s)|\lesssim \Vert l\Vert_{C^0(\partial\mathcal{C}_s)}=O(\Vert f\Vert _{\mathcal{E}_{0}}),
\end{equation*}
from where it follows $ a_2,b_2^i,c_2,d_2=O(\Vert f\Vert _{\mathcal{E}_{0}})$.

\end{proof}

Hence, \begin{equation}\label{equ_l_1}
     e^{2r}\hat{l}_{12},\, e^r\hat{l}_{i3},\, \hat{l}_{33},\,tr(\hat{l})= O(\Vert f\Vert _{\mathcal{E}_{0}}\ww^{-1}_\lambda).
\end{equation}
As a result, the ODEs corresponding to $e^{2r}\hat{l}_{ii}$, $i=1,2$, have the following form \begin{equation*}
    (e^{2r}\hat{l}_{ii})''-2(e^{2r}\hat{l}_{ii})'-\omega e^{2r}\hat{l}_{ii} = O(\Vert f\Vert _{\mathcal{E}_{0}}\ww^{-1}_\lambda).
\end{equation*}
As before,  \begin{equation}\label{equ_l_3}
    e^{2r}\hat{l}_{ii}
    =O(\Vert f\Vert _{\mathcal{E}_{0}}\ww^{-1}_\lambda),\quad i=1,2.
\end{equation}
Finally, combining \eqref{equ_l_1} and \eqref{equ_l_3}, we conclude that 
\begin{equation*}
    \Vert \hat{l}\Vert _{C^0_{\lambda,s}(\mathcal{C}_s)}=O(\Vert f\Vert _{\mathcal{E}_{0}}).
\end{equation*}

 \end{proof}

 Next, by estimating $l-\hat{l}$, we establish a $C^0_{\lambda,s}$ bound for $l$ in the cusp region, using the method presented in Lemma 9.21 of \cite{Hamenstadt-Jackel}. This plus using the lower bound on injectivity radius for the thick part yield the following.

\begin{Lemma}\label{prop_B2}
    Let $\Re(\omega)>\omega_0$. Consider $f\in \mathcal{E}_{0}$ and $l\in C^2(Sym^2(T^*M))\cap H^{1}(M)$ that solves \begin{equation*}
        (\omega I-\Tilde{A}_{h_0})(l)=f.
    \end{equation*}
    Then we have \begin{equation*}
        \Vert l\Vert _{C^0_{\lambda,s}(M)}=O(\Vert f\Vert _{\mathcal{E}_0}).
    \end{equation*}
    where constants depend on $\lambda, s$ and $\omega$.
\end{Lemma}

\begin{proof}
     Let $\Tilde{M}$ be the universal cover of $M$, and let $\Tilde{x}$ be a lift of $x$ in $\Tilde{M}$. Additionally, let $\Tilde{f}:=(\omega I-\Tilde{A}_{h_0})(\Tilde{l})$ and $\Tilde{\hat{f}}:=(\omega I-\Tilde{A}_{h_0})(\Tilde{\hat{l}})$.
    
    Applying the De Giorgi-Nash-Moser estimate \eqref{De Giorgi-Nash-Moser_0} to $(\omega I-\Tilde{A}_{h_0})(\Tilde{l}-\Tilde{\hat{l}})=\Tilde{f}-\Tilde{\hat{f}}$ we obtain the following inequality between $|\Tilde{l}-\Tilde{\hat{l}}|$ and $|\Tilde{f}-\Tilde{\hat{f}}|$
    \begin{equation}\label{eq:De Giorgi-Nash-Moser_Poincare}
       \ww_\lambda(x)|\Tilde{l}-\Tilde{\hat{l}}|(\Tilde{x})\leq C\ww_\lambda(x)\left(\Vert \Tilde{l}-\Tilde{\hat{l}}\Vert _{L^2(B(\Tilde{x}))} +\Vert \Tilde{f}-\Tilde{\hat{f}}\Vert _{L^2(B(\Tilde{x}))}\right),
    \end{equation}
    where $\rho$ is a universal constant, and $C$ depends only on $\omega$. 

    In the cusp $\mathcal{C}_s$, we denote by $i_{\mathcal{C}_s}>0$ a lower bound of the injectivity radius for all points on $\partial\mathcal{C}_s=\cup_kT_k\times \{s\}$. Let $\pi:\Tilde{M}\rightarrow M$ be the universal cover projection. According to (9.40) of \cite{Hamenstadt-Jackel}, for any function $u:M\rightarrow [0,\infty)$, the lift $\Tilde{u}=u\circ\pi$ satisfies 
    \begin{equation*}       \int_{B(\Tilde{x})}\Tilde{u}\dvol_{\Tilde{h}_0}\leq C\frac{1}{i_{\mathcal{C}_s}^2}\int_{B(x)}e^{2r(y)}u(y) \dvol_{h_0}.
    \end{equation*}
    Moreover, the first nonzero eigenvalue of the Laplacian on a flat torus of diameter one satisfies $\lambda_1\geq e^{-2}$(\cite{Gromov2007}, page 250), then we have \begin{equation*}
        \lambda_1(T_k(r))\geq  \frac{1}{\text{diam}(T_k(r))^2}e^{-2}.
    \end{equation*}
    It implies that for any function $u$,
    \begin{equation*}
        \int_{T_k(r)}|u-\hat{u}|^2\dvol_{h_0}\leq e^2 \text{diam}_{h_0}(T_k(r))^2\int_{T_k(r)}|\nabla u|^2\dvol_{h_0}. 
    \end{equation*}
    where $\hat{u} = \frac{1}{\area(T_k)}\int_{T_k}u\dvol_{h_0}$.
    
    Let $D=\text{diam}_{h_0}(T_k(s))$. Substituting $u=\Re(l_{ij})$ and $\Im(l_{ij})$ for $1\leq i,j\leq 3$ and multiplying by $e^{2r(x)}$,
    we have \begin{equation*}
        \int_{T_k(r)}e^{2r}|l-\hat{l}|^2\dvol_{h_0}\lesssim e^2D^2\int_{T_k(r)}|l|^2_{C^1}\dvol_{h_0}.
    \end{equation*}
    On the intrinsic ball $B(x,\rho)\subset \mathcal{C}_s$,  \begin{align*}
        \int_{B(\Tilde{x})}|\Tilde{l}-\Tilde{\hat{l}}|^2\dvol_{\Tilde{h}_0} \leq & C\frac{1}{i_{\mathcal{C}_s}^2}\int_{B(x)}e^{2r(y)}|l-\hat{l}|^2(y)\dvol_{h_0}\\\nonumber
        \lesssim & Ce^2D^2\frac{1}{i_{\mathcal{C}_s}^2}\int_{T_k(r(x),1)}|l|^2_{C^1}\dvol_{h_0}.
    \end{align*}
    where $T_k(r(x),1)$ denotes the region $T_k\times[r(x)-1,r(x)+1]$.

    In order to bound this last integral, we use Proposition~\ref{prop:weightedL2bound} for $\xi<2\lambda$ to obtain
    
    \begin{align}\label{equ_B_2_2}
    \ww^2_\lambda(x) \int_{B(\Tilde{x})}|\Tilde{l}-\Tilde{\hat{l}}|^2\dvol_{\Tilde{h}_0}\lesssim & \ww^2_\lambda(x)\int_{T_k(r(x),1)}|l|^2_{C^1}\dvol_{h_0}\\\nonumber
    \lesssim &\ww^2_\lambda(x)e^{\xi r(x)+\xi}\int_{T_k(r(x),1))}e^{-\xi r(y)}|l|^2_{C^1}(y)\dvol_{h_0}\\\nonumber
    \lesssim& \ww^2_\lambda(x)e^{\xi r(x)+\xi}\Vert f\Vert ^2_{\mathcal{E}_0}\\\nonumber
    = &O(\Vert f\Vert ^2_{\mathcal{E}_0}).
\end{align}
    This provides an estimate for the first term of \eqref{eq:De Giorgi-Nash-Moser_Poincare}.

To estimate the second term of \eqref{eq:De Giorgi-Nash-Moser_Poincare}, observe that since $\hat{f}(y)$ is defined as an average of $f$ in $T_k(r)\ni y$, we have 

\[
    |f-\hat{f}|(y) \lesssim \ww^{-1}_\lambda(y) e^{- r(y)} \Vert f\Vert_{\mathcal{E}_0}.
\] 
We use this to proceed analogously to the bounds on the second term of \eqref{eq:De Giorgi-Nash-Moser_Poincare} and obtain

\begin{align}\label{equ_B_2.3}
    \int_{B(\Tilde{x})}|\Tilde{f}-\Tilde{\hat{f}}|^2\dvol_{\Tilde{h}_0} &\leq C\frac{1}{i_{\mathcal{C}_s}^2}\int_{B(x)}e^{2r(y)}|f-\hat{f}|^2(y)\dvol_{h_0}\\\nonumber
    &\lesssim  \int_{T_k(r(x),1)}e^{2r(y)} \ww^{-2}_\lambda(y) e^{-2 r(y)} \Vert f\Vert^2_{\mathcal{E}_0}\dvol_{h_0}.\\\nonumber
    &\lesssim \ww^{-2}(x) \Vert f \Vert^2_{\mathcal{E}_0}
\end{align}

    Combining \eqref{equ_B_2_2}, \eqref{equ_B_2.3} with \eqref{eq:De Giorgi-Nash-Moser_Poincare}, we obtain \begin{equation*}
        \Vert l-\hat{l}\Vert _{C^0_{\lambda,s}(\mathcal{C}_s)}=\sup_{x\in \mathcal{C}_s}\ww_\lambda(x)|l-\hat{l}|(x)=O(\Vert f\Vert _{\mathcal{E}_{0}}).
    \end{equation*}
    Using Lemma~\ref{claim}, \begin{equation*}
        \Vert l\Vert _{C^0_{\lambda,s}(\mathcal{C}_s)}\leq \Vert \hat{l}\Vert _{C^0_{\lambda,s}(\mathcal{C}_s)}+\Vert l-\hat{l}\Vert _{C^0_{\lambda,s}(\mathcal{C}_s)}=O(\Vert f\Vert _{\mathcal{E}_0}).
    \end{equation*}

    For the thick part $M(s)=M\setminus\mathcal{C}_s$, observe that since we have a lower bound on injectivity radius, \eqref{De Giorgi-Nash-Moser_0} and Proposition~\ref{prop:weightedL2bound} imply $|l(x)|= O(\Vert f\Vert_{L^2(B(\tilde{x}))})$. As $\Vert f\Vert_{L^2(B(\tilde{x},\rho))} = O(\Vert f\Vert_{\mathcal{E}_0})$, the bound $\Vert l\Vert _{C^0_{\lambda,s}(\mathcal{C}_s)}\leq O(\Vert f\Vert _{\mathcal{E}_0})$ follows for points $x$ in the thick part.
    
\end{proof}

For each $l\in\mathcal{E}_1$, it satisfies $l\in C^2(Sym^2(T^*M))\cap H^1(M)$, and $f=(\omega I-\Tilde{A}_{h_0})(l)\in \mathcal{E}_0$, the above lemmas apply. 
Therefore, we have $\Vert l\Vert _{C^0_{\lambda,s}(M)}=O(\Vert f\Vert _{\mathcal{E}_0})$.
Combined with the Schauder estimate \eqref{equ_Schauder}, the estimate \eqref{equ_c_3} follows from
\begin{equation*}
    \Vert l\Vert _{\mathcal{E}_{1}}\lesssim \Vert (\omega I-\Tilde{A}_{h_0})(l)\Vert _{\mathcal{E}_{0}}+\Vert l\Vert _{C^0_{\lambda,s}(M)}\lesssim \Vert (\omega I-\Tilde{A}_{h_0})(l)\Vert _{\mathcal{E}_{0}}.
\end{equation*}

This implies the injectivity of the operator $\omega I-\Tilde{A}_{h_0}$.

\subsubsection{Surjectivity of $\omega I-\Tilde{A}_{h_0}$}\label{step3}
        
We showed that for each $\omega\in\mathbb{C}$ with $\Re(\omega)> \omega_0$, the operator $\omega I-\Tilde{A}_{h_0}$ is an isomorphism. It remains to check the surjectivity. 
Combining Lemmas~\ref{claim} and \ref{prop_B2} and the Schauder estimate \eqref{equ_c_3}, we obtain the following result.  
\begin{Cor}\label{cor_claim}
    Consider 
    \[
    \omega_0 = -\lambda(2-\lambda), \quad \lambda\in (0,1].
    \]
     For each $f\in\mathcal{E}_{0}$, suppose there is $l\in C^2(Sym^2(T^*M))\cap H^{1}(M)$ that solves \begin{equation}\label{equ_w-A=f}
    (\omega I-\Tilde{A}_{h_0})(l)=f,
\end{equation}
where $\Re(\omega)>\omega_0$. Then $l\in \mathcal{E}_{1}$.
\end{Cor}

To complete the proof of surjectivity, we use the method outlined in Proposition 4.7 of \cite{Hamenstadt-Jackel} to demonstrate the existence of the solution.
        
        First, consider a smooth tensor $f\in \mathcal{E}_{0}$ with compact support, and solve \eqref{equ_w-A=f} for $\Re(\omega)>\omega_0$.
        Let \begin{equation*}
            a: H^1_0(Sym^2(T^*M))\times H^1_0(Sym^2(T^*M))\rightarrow \mathbb{R}
        \end{equation*} be the sesquilinear form associated with $\omega I-\Tilde{A}_{h_0}$. 
        We claim that $a$ is coercive. To see this, we decompose $f$ into its trace and trace-free part, specifically $f= \phi h_0+f^0$, where $\phi =\frac13 tr(f)$ and $tr(f^0)=0$. The bilinear forms associated with $\omega-\Tilde{A}_{h_0}$ on functions and $(0,2)$-tensors are both bounded and coercive. 
        
        To be more specific, the sesquilinear form $a_1:H^1(M)\times H^1(M)\rightarrow \mathbb{c}$ corresponding to the trace of $f$ is as follows: 
       \begin{equation*}
           a_1(u,v)=
           \int_{M} \left(\langle \nabla u,\nabla v\rangle +(4+\omega)u\overline{v}\right)\dvol_{h_0},\quad v\in C_c^\infty(M).
        \end{equation*}
        As
        \begin{equation}
            \Re((a_1(u,u)) \geq \Vert \nabla u \Vert^2_{L^2(M)} +  (4+\Re(\omega))\Vert u \Vert^2_{L^2(M)}
        \end{equation}
        then for $\Re(\omega)\geq -4$ we have that $a_1$ is coercive.

        Consequently, the Lax-Milgram theorem applies and gives rise to a unique $u\in H^1(M)$ such that $a_1(u,v)=\langle \phi,v\rangle$ for all test functions $v\in C_c^\infty(M)$. Moreover, Weyl lemma indicates that $u$ is smooth, then it solves $(\omega I-\Tilde{A}_{h_0})(uh_0)=\phi h_0$.     
        
        Furthermore, let $E\rightarrow M$ be the vector bundle of symmetric $(0,2)$-tensors with vanishing trace. The sequilinear form $a_2: H^1_0(E)\times H^1_0(E)\rightarrow \mathbb{C}$ in $a$ is
        \begin{equation*}
            a_2(l,k)=
            \int_{M} \left(\langle \nabla l,\nabla k\rangle + \langle Ric(l),k\rangle + (4+\omega) \langle l,k\rangle \right)\dvol_{h_0}.
        \end{equation*}
        As $Ric(\hat{l}) = -6\hat{l} + 2tr(\hat{l})h_0$ we get
       \begin{equation*}
           \Re(a_2(l,l)) \geq  \Vert \nabla l\Vert ^2_{L^2(M)} -(2-\Re(\omega))\Vert l\Vert ^2_{L^2(M)}.
        \end{equation*}
        Using Poincar\'{e}'s inequality (Proposition 3.1 of \cite{Hamenstadt-Jackel}) for tensors with vanishing trace, we have \begin{equation}\label{equ_a_2_coercive}
            \Re(a_2(l,l))\geq 3\Vert  l\Vert ^2_{L^2(M)}-(2-\Re(\omega))\Vert l\Vert ^2_{L^2(M)} = (\Re(\omega)+1)\Vert l\Vert ^2_{L^2(M)}.
        \end{equation}
        It is coercive for any $\Re(\omega)>-1$.

 Consequently, for any $\omega$ with $\Re(\omega)>\omega_0\geq -1$, we can find a smooth tensor $l^0\in H^1_0(E)$ with $(\omega I-\Tilde{A}_{h_0})(l^0)=f^0$. Hence, $l=uh_0+l^0$ is a smooth tensor vanishing at infinity that solves \eqref{equ_w-A=f}, and it satisfies $l\in C^2(Sym^2(T^*M))\cap H^1(M)$. 
  As demonstrated in Corollary~\ref{cor_claim}, we have $l\in \mathcal{E}_{1}$.

For a general tensor $f$ in $\mathcal{E}_{0}$, it can be approximated by a sequence of smooth tensors $(f_j)_j\subset C_c^\infty(Sym^2(T^*M))\subset \mathcal{E}_{0}$, where each $f_j$ has a smooth solution $l_j\in \mathcal{E}_{1}$. Repeating the process used to prove \eqref{equ_c_3}, we can conclude that $(l_j)_j$ forms a Cauchy sequence in $\mathcal{E}_{1}$. Therefore, the limit $l$ exists and solves $(\omega I-\Tilde{A}_{h_0})(l)=f$. By the completeness of $\mathcal{E}_1$, we find that $l\in \mathcal{E}_{1}$. This proves that $\omega I-\Tilde{A}_{h_0}$ is bijective.

Furthermore, the estimate of the inverse operator follows from equation \eqref{equ_c_3}, thereby completing the proof of Proposition~\ref{prop_C_3}.

\subsubsection{Uniform bound}\label{step4}

 \begin{Prop}\label{prop_C_3_2}
 Let $\omega_1> \omega_0$. Then there exists $C=C(M,h_0,s,\omega_1)$, such that for any $\omega$ with $\Re(\omega)> \omega_1$, we have 
        \begin{equation*}
            \Vert l\Vert _{\mathcal{E}_{1}}\leq C \Vert (\omega I-\Tilde{A}_{h_0})(l)\Vert _{\mathcal{E}_{0}}\quad \forall l\in \mathcal{E}_{1}.
        \end{equation*}
\end{Prop}

\begin{proof}
    We proceed by contradiction. Hence assume that there exists sequences $\omega_n\in\mathbb{R}$, $l_n \in \mathcal{E}_1$ so that
    \[
    \Vert l_n\Vert_{\mathcal{E}_1} = 1,\quad \lim_{n\rightarrow+\infty} \Vert (\omega_nI-\Tilde{A}_{h_0})(l_n)\Vert_{\mathcal{E}_0} = 0,
    \]
    while denoting $(\omega_nI-\Tilde{A}_{h_0})(l_n)$ by $f_n$.
    
    We divide the proof in the following two cases.
    
    \textit{Case 1: $\lim_{n\rightarrow +\infty} |\omega_n| =+\infty$.}
    
    As $\Tilde{{A}}_{h_0}$ is bounded we have $\Vert \Tilde{A}_{h_0} l_n\Vert_{\mathcal{E}_0}$ is a bounded sequence. Hence
    \[
    |\Vert \omega_n l_n \Vert_{\mathcal{E}_0} - \Vert \Tilde{A}_{h_0}(l_n)\Vert_{\mathcal{E}_0} | \leq \Vert (\omega_nI-\Tilde{A}_{h_0})(l_n)\Vert_{\mathcal{E}_0}.
    \]
    Since $\lim_{n\rightarrow +\infty} |\omega_n| =+\infty$ it follows then that $\Vert l_n\Vert_{\mathcal{{E}}_0} \xrightarrow[n\rightarrow +\infty]{} 0$. From this and $\Vert l_n\Vert_{\mathcal{E}_1} = 1$, we have that after possibly passing through a subsequence, on the thick part of $M(s)$, the sequence $l_n$ converges to $0$ in $\mathcal{E}_1$ (and in fact, on any compact subset of $M$). In particular, the equality $\Vert l_n\Vert_{\mathcal{E}_1} = 1$ is realized by taking the $\mathcal{E}_1$ norm restricted to the thin region.

    Then in $M\setminus M(s) =  \cup_k T_k\times [s,+\infty)$, we can take a sequence $x_n \in \cup_k T_k$ and $r_n \in [s, +\infty)$, so that $\ww_\lambda(x_n,r_n)\Vert l_n|_{\Tilde{B}(x_n,r_n)}\Vert_{\mathfrak{h}^{2+\alpha}} \geq \frac12$. As the sequence $l_n$ converges to $0$ in compact sets with respect to $\mathcal{E}_0$, we must have $\lim_{n\rightarrow+\infty} r_n = +\infty$. Consider the sequence
    \[
    L_n(x,r) : = l_n(x,r+r_n) \ww_\lambda (r_n).
    \]
    As one can verify that $\ww_\lambda(r)\ww_\lambda(r') \leq \ww_\lambda(r+r')$, it follows then
    \[
    \Vert L_n\Vert_{\mathcal{E}_1(\cup_k T_k\times [s,+\infty))} \leq 1, \quad \Vert L_n\Vert_{\mathcal{E}_0(\cup_k T_k\times [s,+\infty))} \xrightarrow[n\rightarrow +\infty]{} 0.
    \]
    Hence, after possibly taking a subsequence, we have that $L_n$ must converge to $0$ in compact sets of $\cup_k T_k\times [s,+\infty)$ with respect to $\mathcal{E}_1$. But this is not possible since by construction $\Vert L_n|_{\Tilde{B}(x_n,0)}\Vert_{\mathfrak{h}^{2+\alpha}} = \ww_\lambda(x_n,r_n)\Vert l_n|_{\Tilde{B}(x_n,r_n)}\Vert_{\mathfrak{h}^{2+\alpha}} \geq \frac12$.

    \textit{Case 2: $|\omega_n|$ is bounded.}

    As $\Re(\omega)>\omega_1>\omega_0$ belong to a given compact set, all dependencies on $\omega$ used in the proofs of Lemma~\ref{claim}, Lenmma~\ref{prop_B2} and Corollary~\ref{cor_claim} can be uniformly controlled, which contradicts the assumption in this case.
\end{proof}

        By Propositions~\ref{prop_C_3} and \ref{prop_C_3_2}, the assumption in Remark 1.2.1(a) of \cite{Amann} holds. As a result, it follows from Theorem 1.2.2 that $\Tilde{A}_{h_0}$ is the infinitesimal generator of a strongly continuous analytic semigroup on $\mathcal{L}(\mathcal{E}_{1},\mathcal{E}_{0})$.

\subsection{Condition \ref{(C4)}}
          Recall the fixed numbers $0<\sigma<\rho<1$ in our definition of $\mathcal{X}_{j}$ and $\mathcal{E}_{j}$, $j=0,1$ in \eqref{def_X_i,E_i}. According to Corollary~\ref{cor_reiteration}, if we can find a number $\theta\in (0,1)$, such that \begin{equation*}
            (1-\theta)(0+\sigma)+\theta(2+\sigma)=2\theta+\sigma=\rho,
        \end{equation*}
        then the following isomorphism holds. \begin{equation*}
            (\mathcal{E}_{0},D(\Tilde{A}_{h_0}))_\theta=(\mathcal{E}_{0},\mathcal{E}_{1})_\theta=(\mathfrak{h}^{0+\sigma}_{\lambda,s},\mathfrak{h}^{2+\sigma}_{\lambda,s})_\theta\cong \mathfrak{h}^{0+\rho}_{\lambda,s}=\mathcal{X}_{0}.
        \end{equation*}
        Therefore, the first isomorphism of \eqref{equ_isom} is true if we set $\theta:=\frac{\rho-\sigma}{2}$, which belongs to the interval $(0,1)$, and thus it is well-defined. 

        We now check the second isomorphism of \eqref{equ_isom}. Recall the definition \eqref{def_1+theta}, together with the previous result, it implies that \begin{equation*}
            (\mathcal{E}_{0},D(\Tilde{A}_{h_0}))_{1+\theta}=\{l\in \mathcal{E}_{1}:\Tilde{A}_{h_0}(l)\in\mathcal{X}_{0}\}.
        \end{equation*}
        The space is equipped with the graph norm of $\Tilde{A}_{h_0}$ with respect to $\mathcal{X}_{0}$.
        Thus, this norm is equivalent to $\Vert \cdot\Vert _{\mathcal{X}_{0}}+\Vert \Tilde{A}_{h_0}(\cdot)\Vert _{\mathcal{X}_{0}}$. 

        Furthermore, to incorporate the space $\mathcal{X}_{1}$, we observe that \begin{equation*}
            \mathcal{X}_{1}\subset \{l\in \mathcal{E}_{1}:\Tilde{A}_{h_0}(l)\in\mathcal{X}_{0}\}=(\mathcal{E}_{0},D(\Tilde{A}_{h_0}))_{1+\theta}.
        \end{equation*}
        Therefore, the corresponding norms adhere to the following comparison. \begin{equation}\label{equ_norm_equiv}
            \Vert l\Vert _{\mathcal{X}_{1}}\lesssim \Vert l\Vert _{\mathcal{X}_{0}}+\Vert \Tilde{A}_{h_0}(l)\Vert _{\mathcal{X}_{0}} \quad \forall l\in \mathcal{X}_{1}.
        \end{equation}
        
       Next, by reasoning akin to the proof of \ref{(C3)} using Schauder estimates \eqref{equ_Schauder}, we can derive the converse direction of \eqref{equ_norm_equiv}.  In fact, this is shown in (2.6) of \cite{Hamenstadt-Jackel}, where the classical interior Schauder estimates are applied to 
       \begin{equation*}
            (h_0)_\phi^{pq}\frac{\partial^2 F^m}{\partial x^p\partial x^q}=0, \quad\forall m\in \mathbb{N},
        \end{equation*}
        where $F=\psi\circ \phi^{-1}$, and $\phi, \psi: B(x,r)\subset M\rightarrow \mathbb{R}^3$ are harmonic charts. Combining this with the previous discussion on weights, we obtain
        \begin{equation*}
            \Vert \Tilde{A}_{h_0}(l)\Vert _{\mathcal{X}_{0}}\lesssim \Vert l\Vert _{\mathcal{X}_{1}}.
        \end{equation*}
        It leads to \begin{equation*}
            \Vert l\Vert _{\mathcal{X}_{0}}+\Vert \Tilde{A}_{h_0}(l)\Vert _{\mathcal{X}_{0}}\leq \Vert l\Vert _{\mathcal{X}_{1}}+\Vert \Tilde{A}_{h_0}(l)\Vert _{\mathcal{X}_{0}}\lesssim \Vert l\Vert _{\mathcal{X}_{1}}.
        \end{equation*}
        Consequently, the condition \ref{(C4)} remains valid.

\section{Proof of Theorem~\ref{thm_main}}\label{section_proof}
We now establish the exponential attractivity and complete the proof of Theorem~\ref{thm_main}.

Let $\alpha\in (0,1)\setminus\{\frac{1-\rho}{2},1-\frac{\rho}{2}\}$, by Corollary~\ref{cor_reiteration}, we have $\mathcal{X}_{\alpha}=(\mathcal{X}_{0}, \mathcal{X}_{1})_\alpha \cong \mathfrak{h}_{\lambda,s}^{2\alpha +\rho}$, where $2\alpha +\rho\notin\mathbb{N}$.
Let $\epsilon>0$ be a sufficiently small constant. Applying the stability theorem (Theorem~\ref{thm:Ckstability}) with order $k=3$, for any $\epsilon>0$, there exists a $\delta>0$ so that if $g$ is in the $C^0$ neighbourhood of $h_0$, then $g(t)$ is in the $C^3$ neighbourhood of $h_0$ for any $t\geq 1$. We will denote $h(t):=g\left(t+1\right)$, observing that it is sufficient to prove
\[
    \Vert h(t)-h_0\Vert _{\mathcal{X}_{1}}\leq \frac{c_1}{t^{1-\alpha}} e^{-\omega t}\Vert h(0)-h_0\Vert_{\mathcal{X}_{\alpha}}
\]
as the $C^3$ norm of $h(0)-h_0$ dominates the right hand-side and in turns is dominated by the $C^0$ norm of $g-h_0$ as in Theorem~\ref{thm:Ckstability}.

In particular we have that $A_{h(t)}\in \mathcal{L}(\mathcal{X}_1,\mathcal{X}_0)$, while  \begin{equation*}
    \Vert \nabla^3h(t)\Vert _{C^0(M)}\lesssim 1,\quad\forall t\in (0,1].
\end{equation*}
Thus we have $h(t)\in C^1\left((0,\infty),\mathcal{X}_0\right)\cap C^0\left((0,\infty),\mathcal{X}_1\right)$, and since $\alpha<1$, \begin{equation*}
    \lim_{t\rightarrow 0}t^{1-\alpha}\left(\Vert h'(t)\Vert _{\mathcal{X}_0}+\Vert h(t)\Vert _{\mathcal{X}_1}\right)\lesssim \lim_{t\rightarrow 0}t^{1-\alpha}\Vert h(t)\Vert _{C^3(M)}\lesssim \lim_{t\rightarrow 0}t^{1-\alpha}=0.
\end{equation*}
It shows that $h(t)\in C_\alpha^1\left((0,\infty),\mathcal{X}_0,\mathcal{X}_1\right)$, as defined in \eqref{equ_C_alpha}.

One can easily see that $F(t):=\left(\mathcal{A}(h(t))-A_{h_0}\right)h(t)\in C^0\left((0,\infty),\mathcal{X}_0\right)$. Moreover, \begin{equation*}
    \lim_{t\rightarrow 0}t^{1-\alpha}\Vert F(t)\Vert _{\mathcal{X}_0}\lesssim  \lim_{t\rightarrow 0}t^{1-\alpha}\Vert h(t)\Vert _{\mathcal{X}_1}\lesssim  \lim_{t\rightarrow 0}t^{1-\alpha}\Vert h(t)\Vert _{C^3(M)}\lesssim \lim_{t\rightarrow 0}t^{1-\alpha}=0,
\end{equation*}
which implies $F(t)\in C_\alpha^0\left((0,\infty),\mathcal{X}_0\right)$.

Recall that $A_{h_0}\in \mathcal{M}_\alpha(\mathcal{X}_1,\mathcal{X}_0)$, as established in Section~\ref{section_semigroup generator}.
As a consequence, the maximal regularity property implies that there exists solution $H(t)\in C_\alpha^1\left((0,\infty),\mathcal{X}_0,\mathcal{X}_1\right)$ to the linear equation 
\begin{equation*}
\begin{cases}
    \dfrac{\partial}{\partial t}H(t)=A_{h_0}H(t)+\left(\mathcal{A}(h(t))-A_{h_0}\right)H(t),\\
    H(0)=h(0).
\end{cases}
\end{equation*}
Such solution can be expressed by the integral formula \begin{equation*}
    H(t):=e^{tA_{h_0}} h(0)+\int_0^t e^{(t-s)A_{h_0}}\left(\mathcal{A}(h(s))-A_{h_0}\right)h(s)\,ds,
\end{equation*}
for $t\in [0,\infty)$.

We observe that $h(t)\in C_\alpha^1\left((0,\infty),\mathcal{X}_0,\mathcal{X}_1\right)$ also solves the linear system, and hence $h(t)=H(t)$ for all $t\in [0,\infty)$. In other words, the DeTurck flow $h(t)$ takes the following form.
\begin{equation*}
    h(t)=e^{tA_{h_0}} h(0)+\int_0^t e^{(t-s)A_{h_0}}\left(\mathcal{A}(h(s))-A_{h_0}\right)h(s)\,ds.
\end{equation*}
Let $l(t):=h(t)-h_0$. We obtain 
\begin{equation}\label{equ_l_A_exp}
    l(t)=e^{tA_{h_0}} l(0)+\int_0^t e^{(t-s)A_{h_0}}\left(\mathcal{A}(h(s))(h(s))-\mathcal{A}(h_0)(h_0)-A_{h_0}(l(s))
    \right)\,ds.
\end{equation}

Furthermore, by Lemma~\ref{claim}, the resolvent set of $A_{h_0}$ contains all $\omega\in\mathbb{C}$ with $\Re(\omega)>\omega_0$, where
    $$
    \omega_0= -\lambda(2-\lambda)\quad \forall \lambda\in (0,1].
    $$

Fix an arbitrary real constant $\omega\in (\omega_0,0)$. 
Using the property of the interpolation space $\mathcal{X}_\alpha$ from Definition~\ref{def_K_J_classes} (2), we obtain the following estimate for the first term in \eqref{equ_l_A_exp}. 
\begin{equation}\label{equ_1st}
    t^{1-\alpha}e^{|\omega|t}\Vert e^{tA_{h_0}}l(0)\Vert _{\mathcal{X}_1}\leq t^{-\alpha}K(t,l(0);\mathcal{X}_0,\mathcal{X}_1)\leq c_1(\alpha) \Vert l(0)\Vert _{\mathcal{X}_\alpha},
\end{equation}
where $c_1(\alpha)>0$ is a constant depending on $\alpha$. 

Consider the second term of \eqref{equ_l_A_exp}. For any $\tau\in [0,1]$, $h_\tau(s):=\tau h(s)+(1-\tau)h_0$ remains $\epsilon$-close to $h_0$ in $C^2$. Therefore, a calculation similar to that in Section~\ref{subsection_DA} indicates that the linearization at $h_\tau(s)$, denoted by $A_{h_\tau(s)}$, belongs to $\mathcal{X}_0$, and the map $\tau\mapsto \mathcal{A}(h_\tau(s))(h_\tau(s))$ defined on $[0,1]$ is $C^1$. By applying the mean value theorem to this map we have  \begin{align*}
    &\left\Vert \mathcal{A}(h(s))(h(s))-\mathcal{A}(h_0)(h_0)-A_{h_0}(l(s))
    \right\Vert _{\mathcal{X}_0}\\
    \leq & \max_{0\leq \tau \leq 1}\Vert A_{h_\tau(s)}(l(s))-A_{h_0}(l(s))\Vert _{\mathcal{X}_0}\\
    \leq &\max_{0\leq \tau \leq 1}\Vert A_{h_\tau(s)}-A_{h_0}\Vert _{\mathcal{L}(\mathcal{X}_1,\mathcal{X}_0)}\Vert l(s)\Vert _{\mathcal{X}_1}\\
    \leq & c_2(\epsilon)\Vert l(s)\Vert _{\mathcal{X}_1},
\end{align*}
where $c_2(\epsilon)>0$ is a constant with \begin{equation}\label{c_2(alpha,epsilon)}
    c_2(\epsilon)\rightarrow 0\text{ as }\epsilon\rightarrow 0.
\end{equation}
This is obtained because the normalized Ricci-DeTurck flow $h(t)$ remains in the $\epsilon$-neighborhood of $h_0$ in $C^2$ for all time.

Therefore, we have \begin{align}\label{equ_2nd}
   & t^{1-\alpha}e^{|\omega| t} \left\Vert \int_0^t e^{(t-s)A_{h_0}}\left(\mathcal{A}(h(s))(h(s))-\mathcal{A}(h_0)(h_0)-A_{h_0}(l(s))
    \right)\,ds\right\Vert _{\mathcal{X}_1} \\\nonumber
   \leq & k(|\omega|,\alpha)\sup_{0<\tau\leq t}\tau^{1-\alpha}e^{|\omega|\tau}\left\Vert \mathcal{A}(h(\tau))(h(\tau))-\mathcal{A}(h_0)(h_0)-A_{h_0}(l(\tau))
    \right\Vert _{\mathcal{X}_0}\\\nonumber
    \leq & k(|\omega|,\alpha)c_2(\epsilon) \sup_{0<\tau\leq t}\tau^{1-\alpha}e^{|\omega|\tau} \Vert l(\tau)\Vert _{\mathcal{X}_1}.
\end{align}
The first inequality follows from Proposition 2.3 of \cite{Simonett}, where $k(|\omega|,\alpha)$ is a constant depending only on $|\omega|$ and $\alpha$.
From \eqref{c_2(alpha,epsilon)}, we can choose a sufficiently small $\epsilon$ (so $c_2(\epsilon)$ is also small enough)  such that $k(|\omega|,\alpha)c_2(\epsilon)\leq\frac{1}{2}$. Combining this with the estimates \eqref{equ_l_A_exp}, \eqref{equ_1st}, and \eqref{equ_2nd}, we obtain 
\begin{equation*}
    t^{1-\alpha} e^{|\omega| t} \Vert l(t)\Vert _{\mathcal{X}_1}\leq c_1(\alpha) \Vert l(0)\Vert _{\mathcal{X}_\alpha}+ \frac{1}{2}\sup_{0<\tau\leq t}\tau^{1-\alpha}e^{|\omega|\tau} \Vert l(\tau)\Vert _{\mathcal{X}_1}.
\end{equation*}
Hence \begin{equation*}
    \Vert l(t)\Vert _{\mathcal{X}_1}\leq \frac{c_1(\alpha)}{2t^{1-\alpha}} e^{-|\omega| t}\Vert l(0)\Vert _{\mathcal{X}_\alpha}\lesssim \frac{1}{t^{1-\alpha}}e^{-|\omega|t}\Vert l(0)\Vert _{C^2(M)},\quad \forall t> 0.
\end{equation*}

\bigskip
\appendix
\section{Proof of Proposition~\ref{prop_interp_isom}}\label{appendix_interpolation}
In this appendix, we prove Proposition~\ref{prop_interp_isom}, which is restated as Proposition~\ref{Prop_Lunardi_3.5}.

First, we consider the interpolation spaces between the weighted spaces $C^0_{\lambda,s}$ and $C^1_{\lambda,s}$. The following lemma follows the approach used in Lemma 7.2 of \cite{Wu}.
\begin{Lemma}\label{lemma_wu_7.2}
    For any $\theta\in (0,1)$, we have the isomorphism \begin{equation*}
(C^0_{\lambda,s},C^1_{\lambda,s})_\theta\cong \mathfrak{h}^{\theta}_{\lambda,s}.
    \end{equation*}
\end{Lemma}

\begin{proof}
Let $X=C^0_{\lambda,s}$, $Y=C^1_{\lambda,s}$, and $Z=\mathfrak{h}^{\theta}_{\lambda,s}$. The spaces satisfy $Y\subset Z\subset X$. Denote the norm on the interpolation space $(X,Y)_\theta$ by $\Vert \cdot\Vert _\theta$. When there is no ambiguity, we abbreviate $K(t,h;X,Y)$ as $K(t,h)$.

First, we show that $(X,Y)_\theta\subset Z$. For each $h\in (X,Y)_\theta$, and for each decomposition $h=a+b$, where $a\in X, b\in Y$, we have \begin{equation*}
    \Vert h\Vert _X\leq \Vert a\Vert _X+\Vert b\Vert _X\leq \Vert a\Vert _X+\Vert b\Vert _Y.
\end{equation*}
Since the above decomposition is arbitrary, this implies that \begin{equation}\label{equ_lem01_1}
    \Vert h\Vert _X\leq K(1,h)\leq \sup_{t>0}t^{-\theta}K(t,h)=\Vert h\Vert _\theta.
\end{equation}

Recall that in the definition of weighted little H{\"o}lder spaces in Definition~\ref{def_little_holder},
the weight $\ww_\lambda$ is defined as
\begin{align*}
    \ww_\lambda(x) &=\begin{cases}
    e^{-\lambda r(x)}\quad \lambda\in (0,1),\\
    (r(x)+1)e^{-r(x)}\quad \lambda=1.
\end{cases}
\end{align*}

For any $y_1\neq y_2\in \Tilde{B}(x)$, let $\gamma$ be the geodesic connecting $y_1$ to $y_2$ in $\Tilde{B}(x)$. For any $z\in\gamma$, we have $|r(z)-r(x)|\leq 1$ which in particular implies
$\ww_\lambda(z)^{-1}\leq 2e^\lambda \ww_\lambda(x)^{-1}$. 

For each tensor $h$ on $M$, the lift of $h$ on $\mathbb{H}^3$ is still denoted by $h$.
For an arbitrary decomposition $h=a+b$ as above we can estimate the following 
\begin{align*}
    |a(y_1)-a(y_2)|=& |a(y_1)|+|a(y_2)|\leq 2e^\lambda \ww_\lambda(x)^{-1}\Vert a\Vert _X,\\
    |b(y_1)-b(y_2)|\leq & \int_\gamma |\langle\nabla b,\gamma'\rangle|\leq 2e^\lambda \ww_\lambda(x)^{-1}d(y_1,y_2)\Vert b\Vert _Y.
\end{align*}
Thus, we obtain 
\begin{align*}
    |h(y_1)-h(y_2)| \leq & |a(y_1)-a(y_2)|+ |b(y_1)-b(y_2)|\\
    \leq & 2e^\lambda \ww_\lambda(x)^{-1} \left(\Vert a\Vert _X+d(y_1,y_2)\Vert b\Vert _Y\right).
\end{align*}
As the decomposition $h=a+b$ was arbitrary it follows
\begin{align*}
    |h(y_1)-h(y_2)| \leq & 2e^\lambda \ww_\lambda(x)^{-1}  K\left(d(y_1,y_2),h\right)\\
    \leq & 2e^\lambda \ww_\lambda(x)^{-1} \left(d(y_1,y_2)\right)^\theta \Vert h\Vert _\theta.
\end{align*} 
When it is combined with \eqref{equ_lem01_1}, we get
\begin{equation*}
    \Vert h\Vert _Z\leq \Vert h\Vert _X+\sup_{x\in M,\,y_1\neq y_2\in \Tilde{B}(x)}\ww_\lambda(x)\frac{|h(y_1)-h(y_2)|}{d(y_1,y_2)^\theta}
    \lesssim \Vert h\Vert _\theta.
\end{equation*}
Therefore, we conclude that $(X,Y)_\theta\subset Z$.

\smallskip
Next, we argue that $Z\subset (X,Y)_\theta$. 
For each $h\in Z\subset X$, when $t\in [1,\infty)$, choose the decomposition $h=h+0$, where $h\in X$ and $0\in Y$. We have \begin{equation}\label{equ_lem01_2}
    t^{-\theta}K(t,h)\leq t^{-\theta}\Vert h\Vert _X\leq \Vert h\Vert _X,\quad \forall t\in [1,\infty).
\end{equation}

It remains to consider $t\in (0,1)$. Fix a smooth function $\eta:\mathbb{R}^+_0\rightarrow\mathbb{R}^+_0$ so that for each $x\in M$ and $\Tilde{x}$ a lift of $x$ to $\mathbb{H}^3$. We have that the smooth bump function $\eta_{\Tilde{x}}(\Tilde{y}) : = \eta(d(\Tilde{x},\Tilde{y}))\geq 0$ satisfies \begin{equation}\label{equ_lem01_00}
    \int_{\mathbb{H}^3} \eta_{\Tilde{x}}(\Tilde{y})d\mu_{h_0}(\Tilde{y})=1.
\end{equation}

Moreover, we chose $\eta$ so that $\eta_{\Tilde{x}}$ has compact support contained in $\{\Tilde{y}\in \mathbb{H}^3: d(\Tilde{x},\Tilde{y})< 1\}$.

The geodesic ball $B(\Tilde{x},t)$ in $\mathbb{H}^3$ of radius $t$ has the following volume estimate.  \begin{equation*}
    \vol(B(\Tilde{x},t))=\omega_2\int_0^t \sinh^2(s)ds=\omega_2t^3\left(\frac{1}{3}+\frac{1}{15}t^2+O(t^4)\right),
\end{equation*}
where $\omega_2$ represents the area of Euclidean $2$-sphere of radius $1$.
Therefore, for all $t\in (0,1)$, $\frac{\vol(B(\Tilde{x},t))}{t^3}$ is uniformly bounded from both below and above by positive constants. 
As a result, there exist constants $C_0>c_0>0$ and $C_i>0$, where $1\leq i\leq 3$, such that for all $t\in (0,1)$, 
\begin{align*}
    c_0\leq C_t:=\frac{1}{ t^3}\int_{\{d(\Tilde{x},\Tilde{y})< t\}} \eta\left(\frac{d(\Tilde{x},\Tilde{y})}{t}\right)\,d\mu_{h_0}(\Tilde{y}) \leq  &C_0, \\
    \frac{1}{ t^3}\int_{\{d(\Tilde{x},\Tilde{y})< t\}} \left|\partial_i\eta\left(\frac{d(\Tilde{x},\Tilde{y})}{t}\right)\right|\,d\mu_{h_0}(\Tilde{y}) \leq  &C_i.
\end{align*}

We select a decomposition of $h$ as follows.
\begin{align}\label{equ_lem01_a_b}
    b_t(x):=&\frac{1}{C_t}\frac{1}{ t^3}\int_{\{d(\Tilde{x},\Tilde{y})< t\}} h(\Tilde{y})\eta\left(\frac{d(\Tilde{x},\Tilde{y})}{t}\right)\,d\mu_{h_0}(\Tilde{y}),\\\nonumber
    a_t(x):=& h(x)-b_t(x).
\end{align}
Observe that $b_t$ (and subsequently $a_t$) is well defined since the left-hand side of \eqref{equ_lem01_a_b} does not depend on the lift $\Tilde{x}$ of $x$.

For $\Tilde{y}\neq\Tilde{x}$ with $d(\Tilde{x},\Tilde{y})\leq t<1$, we have $\Tilde{y}\in \Tilde{B}(x)$, and \begin{equation*}
    \frac{|h(\Tilde{x})-h(\Tilde{y})|}{d(\Tilde{x},\Tilde{y})^\theta}\leq \ww_\lambda(x)^{-1}\Vert h\Vert _Z.
\end{equation*}
We use this inequality to estimate $a_t$:
\begin{align*}
    |a_t(x)| =& \frac{1}{C_t}|C_th(x)-C_tb(x)|\\
    \leq & \frac{1}{c_0} \frac{1}{ t^3}\int_{\{d(\Tilde{x},\Tilde{y})< t\}} |h(\Tilde{x})-h(\Tilde{y})|\,\eta\left(\frac{d(\Tilde{x},\Tilde{y})}{t}\right) \,d\mu_{h_0}(\Tilde{y})\\
    \leq & \frac{1}{c_0} \frac{1}{ t^3}\int_{\{0<d(\Tilde{x},\Tilde{y})< t\}} \frac{|h(\Tilde{x})-h(\Tilde{y})|}{d(\Tilde{x},\Tilde{y})^\theta} d(\Tilde{x},\Tilde{y})^\theta \,\eta\left(\frac{d(\Tilde{x},\Tilde{y})}{t}\right) \,d\mu_{h_0}(\Tilde{y})\\
    \leq & \frac{1}{c_0} \frac{1}{ t^{3}}\int_{\{d(\Tilde{x},\Tilde{y})< t\}}  \ww_\lambda(x)^{-1}\Vert h\Vert _Z\, t^\theta \eta\left(\frac{d(\Tilde{x},\Tilde{y})}{t}\right) \,d\mu_{h_0}(\Tilde{y})\\
    = & \frac{1 }{c_0} t^\theta \ww_\lambda(x)^{-1}\Vert h\Vert _Z \left(\frac{1}{t^3}\int_{\{d(\Tilde{x},\Tilde{y})< t\}} \eta\left(\frac{d(\Tilde{x},\Tilde{y})}{t}\right) \,d\mu_{h_0}(\Tilde{y})\right)\\
    \leq & \frac{C_0}{c_0}t^\theta \ww_\lambda(x)^{-1}\Vert h\Vert _Z.
\end{align*}

It follows that \begin{equation}\label{equ_lem01_3}
    \Vert a_t\Vert _X\lesssim t^\theta \Vert h\Vert _Z.
\end{equation}

Next we estimate $\Vert b_t\Vert _Y$. Observe that since $\eta\left( \frac{d(\Tilde{x}, \Tilde{y})}{t}\right)$ has compact support contained in $\{\Tilde{y}\in \mathbb{H}^3: d(\Tilde{x},\Tilde{y})< t\}$ we have
\begin{align*}
    \partial_i b_t(x)= \frac{1}{c_0}\frac{1}{t^4} \int_{\{d(\Tilde{x},\Tilde{y})< t\}} h(\Tilde{y})\partial_i \eta\left(\frac{d(\Tilde{x},\Tilde{y})}{t}\right) \,d\mu_{h_0}(\Tilde{y}).
\end{align*}
Additionally, by Stokes theorem, we have
\begin{equation*}
    \frac{1}{t^3} \int_{\{d(\Tilde{x},\Tilde{y})< t\}} \partial_i\eta\left(\frac{d(\Tilde{x},\Tilde{y})}{t}\right)d\mu_{h_0}(\Tilde{y}) = 0.
\end{equation*} 
Using these facts,
for each $1\leq i\leq 3$ we have 
\begin{align*}
    |\partial_i b_t(x)|\leq & \frac{1}{c_0}\frac{1}{t^4} \left|\int_{\{d(\Tilde{x},\Tilde{y})< t\}} h(\Tilde{y})\partial_i \eta\left(\frac{d(\Tilde{x},\Tilde{y})}{t}\right) \,d\mu_{h_0}(\Tilde{y})\right|\\
    = & \frac{1}{c_0}\frac{1}{t^4} \left|\int_{\{d(\Tilde{x},\Tilde{y})< t\}} (h(\Tilde{y})-h(\Tilde{x}))\partial_i \eta\left(\frac{d(\Tilde{x},\Tilde{y})}{t}\right) \,d\mu_{h_0}(\Tilde{y})\right|\\
    \leq & \frac{1}{c_0}\frac{1}{t^4} \int_{\{0<d(\Tilde{x},\Tilde{y})< t\}} \frac{|h(\Tilde{y})-h(\Tilde{x})|}{d(\Tilde{x},\Tilde{y})^\theta} d(\Tilde{x},\Tilde{y})^\theta\left|\partial_i \eta\left(\frac{d(\Tilde{x},\Tilde{y})}{t}\right)\right|  \,d\mu_{h_0}(\Tilde{y})\\
    \leq & \frac{1}{c_0}t^{\theta-1}  \ww_\lambda(x)^{-1}\Vert h\Vert_Z \left(\frac{1}{t^3}\int_{\{d(\Tilde{x},\Tilde{y})< t\}} \left|\partial_i \eta\left(\frac{d(\Tilde{x},\Tilde{y})}{t}\right)\right|  \,d\mu_{h_0}(\Tilde{y})\right)\\
    \leq & \frac{C_i}{c_0}t^{\theta-1}\ww_\lambda(x)^{-1}\Vert h\Vert _Z.
\end{align*}
Thus, we deduce that \begin{equation}\label{equ_lem01_4}
    \Vert \partial_i b_t\Vert _X\lesssim  t^{\theta-1}\Vert h\Vert _Z.
\end{equation}

Moreover, by \eqref{equ_lem01_3}, we have \begin{equation}\label{equ_lem01_5}
    \Vert b_t\Vert _X=\Vert h-a_t\Vert _X\leq \Vert h\Vert _X+\Vert a_t\Vert _X\lesssim \Vert h\Vert _Z. 
\end{equation}

Combining estimates \eqref{equ_lem01_3}, \eqref{equ_lem01_4}, and \eqref{equ_lem01_5}, we obtain that for each $t\in (0,1)$, \begin{equation*}
    t^{-\theta}K(t,h)\leq t^{-\theta}(\Vert a_t\Vert _X+t\Vert b_t\Vert _Y)=t^{-\theta}\Vert a_t\Vert _X +t^{1-\theta}(\Vert b_t\Vert _X+\max_{1\leq i\leq 3}\Vert \partial_i b_t\Vert _X)\lesssim \Vert h\Vert _Z.
\end{equation*}

We conclude that $h\in (X,Y)_{\theta,\infty}$. By \eqref{equ_lem01_2}, we also have $\lim_{t\rightarrow+\infty}t^{-\theta}K(t,h)=0$, so it remains to check $\lim_{t\rightarrow 0^+}t^{-\theta}K(t,h)=0$.
Since \begin{equation*}
    t^{1-\theta}\Vert b_t\Vert _X \lesssim t^{1-\theta}\Vert h\Vert _Z\rightarrow 0\quad\text{as }t\rightarrow 0^+,
\end{equation*}
it is sufficient to consider $t^{-\theta}\Vert a_t\Vert _X$ and $t^{1-\theta}\Vert \partial_ib_t\Vert _X$.

Suppose that $h\in Z$ is the limit of a sequence of smooth compactly supported tensors $h_n$, $n\in \mathbb{N}$. By smoothness, for each $n$ we have \begin{equation*}
   \sup_{0<d(\Tilde{x},\Tilde{y})\leq t} \ww_\lambda(x)\frac{|h_n(\Tilde{x})-h_n(\Tilde{y})|}{d(\Tilde{x},\Tilde{y})^\theta}
   \leq \Vert h_n\Vert _{C^1}t^{1-\theta}\rightarrow 0\quad \text{as } t\rightarrow 0^+.
\end{equation*}
Taking $n$ sufficiently large so that $\Vert h- h_n\Vert_Z \leq \epsilon$ in $Z=\mathfrak{h}_{\lambda,s}^\theta$ implies \begin{equation*}
     \lim_{t\rightarrow0^+}\sup_{0<d(\Tilde{x},\Tilde{y})\leq t} \ww_\lambda(x)\frac{|h(\Tilde{x})-h(\Tilde{y})|}{d(\Tilde{x},\Tilde{y})^\theta}\lesssim \epsilon 
\end{equation*}
As $\epsilon$ was arbitrary we then know that
\begin{equation*}
     \sup_{0<d(\Tilde{x},\Tilde{y})\leq t} \ww_\lambda(x)\frac{|h(\Tilde{x})-h(\Tilde{y})|}{d(\Tilde{x},\Tilde{y})^\theta}\rightarrow 0 \quad \text{as } t\rightarrow 0^+.
\end{equation*}
Therefore, 
\begin{align*}
    t^{-\theta}\ww_\lambda(x)|a_t(x)| 
    \leq & t^{-\theta}\frac{1}{c_0} \frac{1}{ t^3}\int_{\{0<d(\Tilde{x},\Tilde{y})< t\}} \ww_\lambda(x)\frac{|h(\Tilde{x})-h(\Tilde{y})|}{d(\Tilde{x},\Tilde{y})^\theta} d(\Tilde{x},\Tilde{y})^\theta \,\eta\left(\frac{d(\Tilde{x},\Tilde{y})}{t}\right) \,d\mu_{h_0}(\Tilde{y})\\
    \leq & \frac{1}{c_0}\sup_{0<d(\Tilde{x},\Tilde{y})\leq t} \ww_\lambda(x)\frac{|h(\Tilde{x})-h(\Tilde{y})|}{d(\Tilde{x},\Tilde{y})^\theta} \left(\frac{1}{t^3}\int_{\{d(\Tilde{x},\Tilde{y})< t\}} \eta\left(\frac{d(\Tilde{x},\Tilde{y})}{t}\right) \,d\mu_{h_0}(\Tilde{y})\right)\\
    \leq & \frac{C_0}{c_0} \sup_{0<d(\Tilde{x},\Tilde{y})\leq t} \ww_\lambda(x)\frac{|h(\Tilde{x})-h(\Tilde{y})|}{d(\Tilde{x},\Tilde{y})^\theta} \rightarrow 0\quad \text{as } t\rightarrow 0^+.
\end{align*}
Similarly, $t^{1-\theta}\Vert \partial_ib_t\Vert _X\rightarrow 0$ as $t\rightarrow 0^+$.
This shows that \begin{equation*}
    \lim_{t\rightarrow 0^+}t^{-\theta}K(t,h)=\lim_{t\rightarrow 0^+}\left(t^{-\theta}\Vert a_t\Vert _X+t^{1-\theta}\Vert b_t\Vert _Y\right)=0,
\end{equation*}
which completes the proof of $Z\subset (X,Y)_\theta$.
\end{proof}

\begin{Cor}\label{Cor_interp_C^1_C^2}
    For any $\theta\in (0,1)$, we have the isomorphism \begin{equation*}
(C^1_{\lambda,s},C^2_{\lambda,s})_\theta\cong \mathfrak{h}^{1+\theta}_{\lambda,s}.
    \end{equation*}
\end{Cor}

\begin{proof}
Let $X=C^1_{\lambda,s}$, $Y=C^2_{\lambda,s}$, and $Z=\mathfrak{h}^{1+\theta}_{\lambda,s}$. The spaces satisfy $Y\subset Z\subset X$.

First, we show that $(X,Y)_\theta\subset Z$. Since every 3-manifold is parallelizable, one can extend local coordinate vector fields $\partial_i$ to globally defined differential operators $A_i$ by choosing a global orthonormal frame $\{e_i\}$, where each $A_i$ is defined as taking a derivative along the vector field $e_i$.
For each $1\leq i\leq 3$, we have $A_i\in\mathcal{L}(X,C^0_{\lambda,s})$, and $A_i\in\mathcal{L}(Y,C^1_{\lambda,s})$. By Theorem 1.6 in \cite{Lunardi}, it follows that $A_i\in \mathcal{L}\left((X,Y)_\theta,(C^0_{\lambda,s},C^1_{\lambda,s})_\theta\right)$. In other words, for any $h\in (X,Y)_\theta$, by Lemma~\ref{lemma_wu_7.2} we have \begin{equation*}
    \Vert A_ih\Vert _{\mathfrak{h}^\theta_{\lambda,s}}\lesssim \Vert A_ih\Vert _{(C^0_{\lambda,s},C^1_{\lambda,s})_\theta}\lesssim \Vert h\Vert _{(X,Y)_\theta},\quad \forall 1\leq i\leq 3.
\end{equation*}
As $A_i, 1\leq i\leq 3$ span all directions and we already had $h\in C^1_{\lambda,s}$, the previous inequality imples that $h\in Z$, and therefore we conclude that $(X,Y)_\theta\subset Z$.

Next, we prove that $Z\subset (X,Y)_\theta$. For each $h\in Z$, when $t\in [1,\infty)$, choose the decomposition $h=h+0$. As in the argument from the previous lemma, we obtain $t^{-\theta}K(t,h)\leq \Vert h\Vert _X$.

When $t\in (0,1)$, we define $a_t$ and $b_t$ as in \eqref{equ_lem01_a_b}, so $\Vert a_t\Vert _{C^0_{\lambda,s}}\lesssim t^\theta\Vert h\Vert _Z$ follows suit. To simplify notation we will use $\pm$ to symbolize isometries taking $\Tilde{x}$ to a fixed point $0\in\mathbb{H}^3$. Hence we can estimate
\begin{align*}
    |\partial_ia_t(x)|=& \left|\frac{1}{C_t}\frac{1}{t^3}\int_{\{d(0,\Tilde{y})<t\}}\left(\partial_i h(\Tilde{x})-\partial_i h(\Tilde{x}-\Tilde{y})\right)\eta\left(\frac{\Tilde{y}}{t}\right)\,d\mu_{h_0}(\Tilde{y})\right|\\
    \leq& \frac{1}{c_0}\frac{1}{t^3}\int_{\{0<d(0,\Tilde{y})<t\}} \frac{|\partial_i h(\Tilde{x})-\partial_i h(\Tilde{x}-\Tilde{y})|}{d(\Tilde{x},\Tilde{x}-\Tilde{y})^\theta}d(\Tilde{x},\Tilde{x}-\Tilde{y})^\theta\eta\left(\frac{\Tilde{y}}{t}\right)\,d\mu_{h_0}(\Tilde{y})\\
    \leq& \frac{1}{c_0}t^\theta \ww_\lambda(x)^{-1}\Vert h\Vert _Z \left(\frac{1}{t^3}\int_{\{d(0,\Tilde{y})<t\}}\eta\left(\frac{\Tilde{y}}{t}\right)\,d\mu_{h_0}(\Tilde{y})\right)\\
    \leq & \frac{C_0}{c_0}t^\theta \ww_\lambda(x)^{-1}\Vert h\Vert _Z.
\end{align*}
This implies that \begin{equation}\label{equ_lemC12_01}
    \Vert a_t\Vert _X\lesssim t^\theta \Vert h\Vert _Z.
\end{equation}

Furthermore, for $b_t$, we have $\Vert b_t\Vert _{C^0_{\lambda,s}}\lesssim \Vert h\Vert _Z$, and \begin{equation*}
    |\partial_i b_t|=\left|\frac{1}{C_t}\frac{1}{t^3}\int_{\{d(0,\Tilde{y})<t\}}\partial_i h(\Tilde{x}-\Tilde{y})\eta\left(\frac{\Tilde{y}}{t}\right)\,d\mu_{h_0}(\Tilde{y})\right|\lesssim \ww_\lambda(x)^{-1}\Vert h\Vert _Z.
\end{equation*}
Additionally, 
\begin{align*}
    |\partial_{ij}b_t|=& \left|\frac{1}{C_t}\frac{1}{t^4}\int_{\{d(\Tilde{x},\Tilde{y})<t\}}\partial_jh(\Tilde{y})\partial_i\eta\left(\frac{d(\Tilde{x},\Tilde{y})}{t}\right)\,d\mu_{h_0}(\Tilde{y})\right|\\
    \leq &\frac{1}{c_0}\frac{1}{t^4}\int_{\{0<d(\Tilde{x},\Tilde{y})<t\}}\frac{|\partial_j h(\Tilde{y})-\partial_j h(\Tilde{x})|}{d(\Tilde{x},\Tilde{y})^\theta}d(\Tilde{x},\Tilde{y})^\theta \left|\partial_i\eta\left(\frac{d(\Tilde{x},\Tilde{y})}{t}\right)\right| \,d\mu_{h_0}(\Tilde{y})\\
    \leq &\frac{C_i}{c_0}t^{\theta-1}\ww_\lambda(x)^{-1}\Vert h\Vert _Z.
\end{align*}
Therefore, it follows that \begin{equation}\label{equ_lemC12_02}
    \Vert b_t\Vert _Y\lesssim t^{\theta-1}\Vert h\Vert _Z.
\end{equation}

The estimate \eqref{equ_lemC12_01}, along with \eqref{equ_lemC12_02}, implies that for each $t\in (0,1)$, \begin{equation*}
    t^{-\theta}K(t,h)\leq t^{-\theta}\Vert a_t\Vert _X+t^{1-\theta}\Vert b_t\Vert _Y\lesssim \Vert h\Vert _Z.
\end{equation*}
Moreover, analogous to the previous lemma, \begin{equation*}
    \lim_{t\rightarrow\infty}t^{-\theta}K(t,h)=\lim_{t\rightarrow0^+}t^{-\theta}K(t,h)=0.
\end{equation*}
Consequently, we have $h\in (X,Y)_\theta$, and thus $Z\subset (X,Y)_\theta$.
   
\end{proof}

Furthermore, to investigate the interpolation space between $C^0_{\lambda,s}$ and $C^m_{\lambda,s}$, with $m\geq 2$, we present the following lemma. 

\begin{Lemma}\label{lemma_lunardi_3.4}
    $$C^1_{\lambda,s}\in J_{\frac12}(C^0_{\lambda,s},C^2_{\lambda,s})\cap K_{\frac12}(C^0_{\lambda,s},C^2_{\lambda,s}).$$
\end{Lemma}

\begin{proof}
To show $C^1_{\lambda,s}\in J_{\frac12}(C^0_{\lambda,s},C^2_{\lambda,s})$, 
for each weight index $\lambda\in (0,1]$, we take $X=C^0_{\lambda,s}$ and consider the global orthonormal frame $\lbrace e_1, e_2, e_3\rbrace$ and the differential operators $A_i: D(A_i)\subset X\rightarrow X$ ($1\leq i\leq 3$) from Corollary~\ref{Cor_interp_C^1_C^2}.
Let $h\in X$, and for $x\in M$ consider the integral ray $\gamma_i$ with $\gamma_i(0)=x$ and $\gamma_i'(0)=e_i$.

As $\gamma_i$ has unit speed, for any $\omega>\lambda$ we have that the path integral $\int_{0}^\infty e^{-\omega t}h(\gamma_i(t))\,dt$ is finite and the functions $x\mapsto \int_{0}^\infty e^{-\omega t}h(\gamma_i(t))\,dt$ is in $C^0_{\lambda,s}$. Moreover, it is easy to verify that by construction this is precisely $R(\omega,A_i)h$. Hence
\begin{align*}
    \left|(R(\omega,A_i)h)(x)\right| =& \left|\int_{0}^\infty e^{-\omega t}h(\gamma_i(t))\,dt\right|\\
    \leq & \int_{0}^\infty e^{-\omega t} \ww_\lambda(\gamma_i(t))^{-1}|h(\gamma_i(t))|_{X} \,dt\\
    \leq & \int_{0}^\infty e^{-\omega t} e^{\lambda t}\ww_\lambda(x)^{-1}|h(\gamma_i(t))|_{X} \,dt\\
    \leq & \ww_\lambda(x)^{-1}\Vert h\Vert _{X}\int_{0}^\infty e^{-(\omega-\lambda)t}dt\\
    =&\frac{1}{\omega-\lambda} \ww_\lambda(x)^{-1}\Vert h\Vert _{X} <\infty, \quad \forall \omega>\lambda.
\end{align*}

This implies that 
\begin{equation}\label{equ_01}
    \Vert R(\omega,A_i)h\Vert _X
    \leq \frac{1}{\omega-\lambda}\Vert h\Vert _X, \quad \forall \omega>\lambda, 1\leq i\leq 3.
\end{equation}

First, we prove $D(A_i)\in J_{\frac12}\left(X,D(A_i^2)\right)$. For each $h\in D(A_i)$, by \eqref{equ_01}, we have 
\begin{equation*}
    \Vert R(\omega,A_i)A_ih\Vert _X\leq \frac{1}{\omega-\lambda}\Vert A_ih\Vert _X\rightarrow 0 \quad \text{as } \omega\rightarrow\infty.
\end{equation*}
It follows that \begin{equation*}
    \lim_{\omega\rightarrow\infty}\omega R(\omega,A_i)h= \lim_{\omega\rightarrow\infty}R(\omega,A_i)A_ih+h=h, \quad h\in D(A_i).
\end{equation*}
Define $f(\sigma):=\sigma R(\sigma, A_i)h$ for $\sigma>\lambda$, considering $f(\infty)=h$. Moreover,
\begin{align*}
    f'(\sigma)=& R(\sigma,A_i)h-\sigma R(\sigma,A_i)^2h\\
    =& R(\sigma,A_i)(I-\sigma R(\sigma,A_i))h=-R(\sigma,A_i)^2A_i h.
\end{align*}
Therefore for $h\in D(A_i)$, \begin{equation*}
    h-\omega R(\omega,A_i)h=-\int_\omega^\infty R(\sigma,A_i)^2 A_ih\,d\sigma,\quad \omega>\lambda.
\end{equation*}
Similarly, for $h\in D(A_i^2)$, 
\begin{equation}\label{equ_02}
    A_ih=\omega A_iR(\omega,A_i)h-\int_\omega^\infty R(\sigma,A_i)^2 A_i^2h\,d\sigma,\quad \omega>\lambda.
\end{equation}

Using \eqref{equ_01}, we obtain the following estimate for the first term of \eqref{equ_02}.
\begin{align}\label{equ_03}
    \Vert A_iR(\omega,A_i)h\Vert_X= & \Vert (\lambda-1)h+(\omega-\lambda)R(\omega,A_i)h\Vert _X \\\nonumber
    \leq & (1-\lambda)\Vert h\Vert _X+ (\omega-\lambda)\Vert R(\omega,A_i)h\Vert _X\\\nonumber
    \leq & (2-\lambda)\Vert h\Vert _X.
\end{align}

To estimate the second term in \eqref{equ_02}, we apply \eqref{equ_01} twice:
\begin{equation}\label{equ_04}
    \left\Vert \int_\omega^\infty R(\sigma,A_i)^2 A_i^2h\,d\sigma \right\Vert _X\leq\Vert A_i^2h\Vert _X\int_\omega^\infty \frac{1}{(\omega-\lambda)^2}d\sigma =\frac{1}{\omega-\lambda}\Vert A_i^2h\Vert _X.
\end{equation}

Substituting \eqref{equ_03} and \eqref{equ_04} into \eqref{equ_02}, we obtain 
\begin{align*}
    \Vert A_ih\Vert _X\leq & \omega(2-\lambda)\Vert h\Vert _X+\frac{1}{\omega-\lambda}\Vert A_i^2h\Vert _X\\
    =& (\omega-\lambda) (2-\lambda)\Vert h\Vert _X+\frac{1}{\omega-\lambda}\Vert A_i^2h\Vert _X +\lambda(2-\lambda)\Vert h\Vert _X,
\end{align*}
for all $\omega>\lambda$. Thus,\begin{equation*}
    \Vert A_ih\Vert _X\leq 2(2-\lambda)^{\frac12}\Vert h\Vert _X^\frac12\Vert A_i^2h\Vert _X^{\frac12} +\lambda(2-\lambda)\Vert h\Vert _X.
\end{equation*}
Recall the graph norm $\Vert h\Vert _{D(A_i)}=\Vert h\Vert _X+\Vert A_i h\Vert _X$, $\Vert h\Vert _{D(A_i^2)}=\Vert h\Vert _X+\Vert A_i^2 h\Vert _X$. As $\lambda\in(0,1]$ we have 
\begin{equation*}
    \Vert h\Vert _{D(A_i)} \leq 2(2-\lambda)^{\frac12}\Vert h\Vert _X^\frac12 \left(\Vert A_i^2h\Vert _X^{\frac12}+\Vert h\Vert _X^\frac12\right)\leq 2\left(2(2-\lambda)\right)^\frac12 \Vert h\Vert _X^\frac12 \Vert h\Vert _{D(A_i^2)}^\frac12.
\end{equation*}
Therefore, $D(A)\in J_{\frac12}\left(X,D(A_i^2)\right)$.

Since $C^1_{\lambda,s}=\cap_{1\leq i\leq 3}D(A_i)$ and $C^2_{\lambda,s}=\subset\cap_{1\leq i\leq 3}D(A_i^2)$, for each $h\in C^2_{\lambda,s}$, it follows that \begin{equation*}
    \Vert h\Vert _{C^1_{\lambda,s}}=\max_{1\leq i\leq 3} \Vert h\Vert _{D(A_i)} \lesssim \max_{1\leq i\leq 3}\Vert h\Vert _{C^0_{\lambda,s}}^\frac12 \Vert h\Vert _{D(A_i^2)}^\frac12\lesssim \Vert h\Vert _{C^0_{\lambda,s}}^\frac12 \Vert h\Vert _{C^2_{\lambda,s}}^\frac12.
\end{equation*}
We conclude that $C^1_{\lambda,s}\in J_{\frac12}(C^0_{\lambda,s},C^2_{\lambda,s})$.

\smallskip

Next, we argue that $C^1_{\lambda,s}\in K_{\frac12}(C^0_{\lambda,s},C^2_{\lambda,s})$. 
For each $h\in C^1_{\lambda,s}$ and $t\in (0,1)$, recall the decomposition $h=a_t+b_t$ defined in \eqref{equ_lem01_a_b}. Note that the estimates for $a_t$ and $b_t$ also applies to $\theta=1$, so we have \begin{equation}\label{equ_lem02_01}
    \Vert a_t\Vert _{C^0_{\lambda,s}}\lesssim t\Vert h\Vert _{C^1_{\lambda,s}},\quad \Vert b_t\Vert _{C^1_{\lambda,s}}\lesssim \Vert h\Vert _{C^1_{\lambda,s}}.
\end{equation}
Furthermore,
\begin{align*}
    |\partial_{ij}b_t(\Tilde{x})|&= \left|\frac{1}{C_t}\frac{1}{t^4}\int_{\{d(\Tilde{x},\Tilde{y})<t\}}\partial_jh(\Tilde{y})\partial_i\eta\left(\frac{d(\Tilde{x},\Tilde{y})}{t}\right)\,d\mu_{h_0}(\Tilde{y})\right|\\
    &\leq \frac{1}{c_0}\frac{1}{t^4} \ww_\lambda(x)^{-1}\Vert h\Vert_{C^1_{\lambda,s}}\int_{\{d(\Tilde{x},\Tilde{y})<t\}} \left|\partial_i\eta\left(\frac{d(\Tilde{x},\Tilde{y})}{t}\right)\right|\,d\mu_{h_0}(\Tilde{y})\\
    &\lesssim t^{-1}\Vert h\Vert_{C^1_{\lambda,s}}.
\end{align*}
It implies \begin{equation}\label{equ_lem02_02}
    \Vert \partial_{ij}b_t\Vert _{C^0_{\lambda,s}}\lesssim \frac{1}{t} \Vert h\Vert _{C^1_{\lambda,s}}.
\end{equation}
By substituting $t$ with $\tau = t^{\frac{1}{2}}$, and applying \eqref{equ_lem02_01} and \eqref{equ_lem02_02}, we have 
\begin{align}\label{equ_07}
    K(t,h;C^0_{\lambda,s},C^2_{\lambda,s})\leq &\Vert a_\tau\Vert _{C^0_{\lambda,s}}+t\Vert b_\tau\Vert _{C^2_{\lambda,s}}\lesssim \tau\Vert h\Vert _{C^1_{\lambda,s}}+t\left(\Vert h\Vert _{C^1_{\lambda,s}}+\frac1\tau \Vert h\Vert _{C^1_{\lambda,s}}\right)\\\nonumber
    \lesssim &t^{\frac12}\Vert h\Vert _{C^1_{\lambda,s}},\quad \forall t\in (0,1).
\end{align}

For $t\in [1,\infty)$, decompose $h=a+b$, where $a=h$ and $b=0$, we have \begin{equation}\label{equ_08}
     K(t,h;C^0_{\lambda,s},C^2_{\lambda,s})\leq \Vert h\Vert _{C^0_{\lambda,s}}\leq t^\frac12\Vert h\Vert _{C^1_{\lambda,s}},\quad \forall t\in (1,\infty).
\end{equation}

Combining \eqref{equ_07} and \eqref{equ_08},
\begin{equation*}
     K(t,h;C^0_{\lambda,s},C^2_{\lambda,s})\lesssim t^\frac12 \Vert h\Vert _{C^1_{\lambda,s}},\quad \forall h\in C^1_{\lambda,s}, \,\forall t>0.
\end{equation*}
Therefore, $C^1_{\lambda,s}\in K_{\frac12}(C^0_{\lambda,s},C^2_{\lambda,s})$.
\end{proof}

\begin{Prop}\label{Prop_Lunardi_3.5}
    For $\theta\neq \frac12$, we have \begin{equation*}
       (C_{\lambda,s}^0, C_{\lambda,s}^2)_\theta\cong \mathfrak{h}_{\lambda,s}^{2\theta}.
    \end{equation*}
     More generally, for any $\theta\in (0,1)$ and $m\in\mathbb{N}$ with $m\theta\notin\mathbb{N}$, \begin{equation*}
    (C_{\lambda,s}^0, C_{\lambda,s}^m)_\theta\cong \mathfrak{h}_{\lambda,s}^{m\theta}.
\end{equation*}
\end{Prop}

\begin{proof}
Consider $\theta\in (0,\frac12)$.
    Note that $C^0_{\lambda,s}\in J_0(C^0_{\lambda,s},C^2_{\lambda,s})\cap K_0(C^0_{\lambda,s},C^2_{\lambda,s})$. According to Lemma~\ref{lemma_lunardi_3.4}, $C^1_{\lambda,s}\in J_{\frac12}(C^0_{\lambda,s},C^2_{\lambda,s})\cap K_{\frac12}(C^0_{\lambda,s},C^2_{\lambda,s})$. Applying the Reiteration Theorem (Theorem~\ref{thm_reiteration}) to $\mathcal{X}_0=C^0_{\lambda,s}$, $\mathcal{X}_1=C^2_{\lambda,s}$, and $\mathcal{E}_0=C^0_{\lambda,s}$ and $\mathcal{E}_1=C^1_{\lambda,s}$, we get \begin{equation}\label{equ_prop5_01}
        (C^0_{\lambda,s}, C^1_{\lambda,s})_{2\theta}\cong (C^0_{\lambda,s},C^2_{\lambda,s})_\theta.
    \end{equation}

    Next, for $\theta\in (\frac12,1)$, consider $C^1_{\lambda,s}\in J_{\frac12}(C^0_{\lambda,s},C^2_{\lambda,s})\cap K_{\frac12}(C^0_{\lambda,s},C^2_{\lambda,s})$ and $C^2_{\lambda,s}\in J_1(C^0_{\lambda,s},C^2_{\lambda,s})\cap K_1(C^0_{\lambda,s},C^2_{\lambda,s})$. Applying the Reiteration Theorem (Theorem~\ref{thm_reiteration}) to $\mathcal{X}_0=C^0_{\lambda,s}$, $\mathcal{X}_1=C^2_{\lambda,s}$, and $\mathcal{E}_0=C^1_{\lambda,s}$ and $\mathcal{E}_1=C^2_{\lambda,s}$, we deduce that \begin{equation}\label{equ_prop5_02}
        (C^1_{\lambda,s}, C^2_{\lambda,s})_{\alpha}\cong(C^0_{\lambda,s},C^2_{\lambda,s})_{\frac{\alpha+1}{2}},
    \end{equation}
    where $\alpha+1=2\theta$.

Combining Lemma~\ref{lemma_wu_7.2}, Corollary~\ref{Cor_interp_C^1_C^2}, and isomorphisms \eqref{equ_prop5_01}, \eqref{equ_prop5_02}, we arrive at the following conclusion. 
$$
 (C_{\lambda,s}^0, C_{\lambda,s}^2)_\theta \cong 
 \begin{cases}
      (C_{\lambda,s}^0, C_{\lambda,s}^1)_{2\theta}\cong \mathfrak{h}_{\lambda,s}^{2\theta}\quad \text{ if } \theta\in (0,\frac12),\\
      (C_{\lambda,s}^1, C_{\lambda,s}^2)_{2\theta-1}\cong \mathfrak{h}_{\lambda,s}^{2\theta}\quad \text{ if } \theta\in (\frac12,1)
 \end{cases}
$$
The general case for $m\in\mathbb{N}$ is obtained by iterating the same process.
\end{proof}

\section{A priori weighted $L^2$ bounds}\label{sec:L2bounds}

Here we prove a priori weighted $L^2$ bound of a tensor $l$ in terms of $f = (\omega I - \Tilde{A}_{h_0})l$ for $\Re(\omega) > -1$. This follows \cite[Section 3.2]{Hamenstadt-Jackel}.

Observe that for cusped hyperbolic manifold $(M,h_0)$, from the inequality
\[
0\leq \Vert \nabla h\Vert^2_{L^2(M)} + \frac12 \langle Ric(h), h \rangle_{L^2(M)}
\]
for real-valued $(0,2)$-symmetric tensors, it follows
\begin{align}\label{eq:L2LaplaceRicciineq}
    0&\leq \Vert \nabla h\Vert^2_{L^2(M)} + \frac12 \Re\langle Ric(h), h \rangle_{L^2(M)}\\\nonumber &= -\Re\langle \Delta h, h \rangle_{L^2(M)} + \frac12\Re\langle Ric(h), h \rangle_{L^2(M)},
\end{align}
where now we consider $h$ to have complex coefficients.

Let then $l$ be a $C^2$ complex-valued $(0,2)$-symmetric tensor, and let
\[
f : = -\Delta l + Ric(l) + (4+\omega) l,
\]
where $\omega\in\mathbb{C}$ satisfies $\Re(\omega)>-1$. Following the implementation of \cite[Corollary 2, Section 3]{Tian} done in \cite[Section 3.2]{Hamenstadt-Jackel}, we prove the following proposition.

\begin{Prop}\label{prop:1staprioribound}
    Let $\varphi \in C^\infty(M)$ be so that $\varphi l$, $\varphi f \in L^2(M)$. Then
    \begin{equation*}
        (1+\Re(\omega))\int_M \varphi^2 |l|^2 \dvol \leq 2\int_M \varphi^2 \Re\langle l,f\rangle \dvol + \int_M |\nabla\varphi|^2|l|^2\dvol.
    \end{equation*}
\end{Prop}
\begin{proof}
    Let $\hat{l} = \varphi l$. In analogy to \cite[Proposition 3.4]{Hamenstadt-Jackel} one establishes
    \begin{align*}
        -\Delta \hat{l} =& -(\Delta\varphi)l -2tr^{(1,4)}(\nabla\varphi\otimes \nabla l) - \varphi\Delta l\\
        =& -(\Delta\varphi)l -2tr^{(1,4)}(\nabla\varphi\otimes \nabla l) + 2\varphi f - (4+\omega)\hat{l} - Ric(\hat{l}).
    \end{align*}
    In particular,
    \begin{align}\label{eq:aux1}
        -\Re\langle\Delta \hat{l},\hat{l}\rangle =& -\Re\langle(\Delta\varphi)l ,\hat{l}\rangle -2\Re\langle tr^{(1,4)}(\nabla\varphi\otimes \nabla l),\hat{l}\rangle \\\nonumber &+ 2\varphi^2\Re \langle f,l\rangle - (4+\Re(\omega))|\hat{l}|^2 - \Re \langle Ric(\hat{l}),\hat{l}\rangle.
    \end{align}
    Defining $\eta = \varphi|l|^2d\varphi$, we get
    \begin{equation}\label{eq:aux2}
        -\nabla^*\eta = |\nabla\varphi|^2|l|^2 + 2\Re\langle tr^{(1,4)}(\nabla\varphi\otimes \nabla l), \hat{l}\rangle +\Re\langle(\Delta\varphi)l,\hat{l}\rangle.
    \end{equation}
    Applying then \eqref{eq:L2LaplaceRicciineq} for $h=\hat{l}$, \eqref{eq:aux1} and $\int_M \nabla^*\eta\dvol$ in \eqref{eq:aux2} we obtain
    \begin{align*}
        0&\leq -\Re\langle \Delta \hat{l}, \hat{l} \rangle + \frac12\Re\langle Ric(\hat{l}), \hat{l} \rangle\\\nonumber
        & = \Vert |\nabla\varphi|l\Vert^2_{L^2(M)} + 2\Re\langle \hat{l},\varphi f\rangle_{L^2(M)} -(4+\Re(\omega))\Vert \hat{l}\Vert^2_{L^2(M)} - \frac12 \Re\langle Ric(\hat{l}), \hat{l} \rangle_{L^2(M)}.
    \end{align*}
    As $Ric(\hat{l}) = -6\hat{l} + 2tr(\hat{l})h_0$ we have that $- \frac12 \Re\langle Ric(\hat{l}), \hat{l} \rangle \leq 3 |\hat{l}|^2$ this with the previous inequality yield
    \begin{equation*}
        (1+\Re(\omega))\Vert \hat{l}\Vert^2_{L^2(M)} \leq \Vert\nabla\varphi|l\Vert^2_{L^2(M)} + 2\Re\langle \hat{l},\varphi f\rangle_{L^2(M)},
    \end{equation*}
    from where the inequality follows for $\varphi$ compactly supported. For general $\varphi$ one can argue as in \cite[Proposition 3.4]{Hamenstadt-Jackel}, so we omit the proof.
\end{proof}
As done in \cite[Corollary 3.5]{Hamenstadt-Jackel} we can substitute $\varphi = e^{-\xi r_x}$ in Proposition~\ref{prop:1staprioribound} to obtain

\begin{equation}\label{eq:expweightProp}
    (1+\Re(\omega))\int_M e^{-2\xi r_x}|l|^2 \dvol \leq 2\int_M e^{-2\xi r_x} \Re\langle l,f\rangle \dvol + \xi^2\int_M e^{-2\xi r_x}|l|^2\dvol.
\end{equation}

\begin{Prop}\label{prop:weightedL2bound}
    Let $\xi < \sqrt{1+\Re(\omega)}$. Then there exists $C=C(\omega,\xi)>0$ so that
    \begin{equation*}
        \int_M e^{-2\xi r_x} \left(|l|^2 + |\nabla l|^2 + |\Delta l|^2 \right)\dvol \leq C \int_M e^{-2\xi r_x} |f|^2\dvol.
    \end{equation*}
\end{Prop}
\begin{proof}
    Substituting $\xi$ in \eqref{eq:expweightProp} and denoting by $\delta = 1+\Re(\omega) -\xi^2$ we obtain
    \begin{equation*}
        \int_M e^{-2\xi r_x}|l|^2 \dvol \leq \frac{2}{\delta}\int_M e^{-2\xi r_x} \Re\langle l,f\rangle \dvol.
    \end{equation*}
    By Cauchy-Schwartz
    \begin{equation*}
        \int_M e^{-2\xi r_x}|l|^2 \dvol \leq \frac{1}{\delta}\int_M e^{-2\xi r_x} \left( \frac{\delta}{2}|l|^2 + \frac{2}{\delta}|f|^2 \right) \dvol,
    \end{equation*}
    from where it follows
    \begin{equation*}
        \int_M e^{-2\xi r_x}|l|^2 \dvol \leq \frac{4}{\delta^2}\int_M e^{-2\xi r_x} |f|^2 \dvol.
    \end{equation*}
    As we can write $\Delta l = -f + Ric(l) + (4+\omega)l$, then it follows
    \begin{equation*}
        \int_M e^{-2\xi r_x}|\Delta l|^2 \dvol \leq C\int_M e^{-2\xi r_x} |f|^2 \dvol.
    \end{equation*}
    For the gradient term $|\nabla l|$ we have
    \begin{equation*}
        -\frac12 \Delta\left( |l|^2\right) = -\Re \langle \Delta l,l \rangle - |\nabla l|^2 \leq \frac12|\Delta l|^2 + \frac12|l|^2 - |\nabla l|^2,
    \end{equation*}
    from where we can proceed as in the later part of Step 1 of \cite[Proposition 4.3]{Hamenstadt-Jackel} to conclude
    \begin{equation*}
        \int_M e^{-2\xi r_x}|\nabla l|^2 \dvol \leq C\int_M e^{-2\xi r_x} |f|^2 \dvol.
    \end{equation*}
    Hence the result follows.
\end{proof}

\bibliographystyle{plain} 
\bibliography{ref}   
\end{document}